\documentclass[11pt, reqno]{amsart}

\usepackage{enumerate}
\usepackage{amsmath}%
\usepackage{amsfonts}%
\usepackage{amssymb}%
\usepackage{graphicx}
\usepackage{mathrsfs}
\usepackage{hyperref}
\usepackage[margin=1.2in]{geometry}
 \usepackage{amsaddr}
\newtheorem{theorem}{Theorem}
\theoremstyle{plain}

\newtheorem{claim}[theorem]{Claim}

\newtheorem{conjecture}[theorem]{Conjecture}
\newtheorem{corollary}[theorem]{Corollary}

\newtheorem{lemma}[theorem]{Lemma}

\newtheorem{problem}[theorem]{Problem}
\newtheorem{proposition}[theorem]{Proposition}
\newtheorem{prop}[theorem]{Proposition}

\numberwithin{equation}{section}
\numberwithin{theorem}{section}
\numberwithin{case}{section}

\numberwithin{subcase}{case}
\def \e{\epsilon}

\def\F{\mathcal{F}}

\def \a{\alpha}
\def\eps{\varepsilon}
\def \e{\epsilon}
\def \r{\gamma}
\def \cP{\mathcal{P}}

\def \bfi{\mathbf{i}}
\def \bfu{\mathbf{u}}
\def \bfv{\mathbf{v}}
\def \bfw{\mathbf{w}}
\def \pack{\textbf{Pack}}
\def \hcf{\mathrm{hcf}}
\def \PM{\textbf{PM}}
\def\a{\alpha}
\def\b{\beta}
\def\d{\delta}

\def\COMMENT#1{}
\let\COMMENT=\footnote
\begin{document}

\title{The complexity of perfect matchings and packings in dense hypergraphs}
\author{Jie Han and Andrew Treglown}
\thanks{JH supported by FAPESP (2013/03447-6, 2014/18641-5, 2015/07869-8), AT  supported by EPSRC grant EP/M016641/1}

\address{School of Mathematics, University of Birmingham, Birmingham, B15 2TT, UK. Email: jhan@ime.usp.br, a.c.treglown@bham.ac.uk}

\begin{abstract}
Given two $k$-graphs $H$  and $F$, a perfect $F$-packing in $H$ 
is a collection of vertex-disjoint copies of $F$ in $H$ which together cover all the vertices in $H$. In the case when $F$ is a single edge, a perfect $F$-packing is simply  a perfect matching. For a given fixed $F$, it is often the case that the decision problem whether an $n$-vertex $k$-graph $H$ contains a perfect $F$-packing is NP-complete. Indeed, if $k \geq 3$, the corresponding problem for perfect matchings is NP-complete~\cite{karp, garey} whilst if $k=2$ the problem is NP-complete in the case when $F$ has a component consisting of at least
$3$ vertices~\cite{HeKi}.

In this paper we give a general tool which can be used to 
determine classes 
of (hyper)graphs for which the corresponding decision problem for perfect $F$-packings is polynomial time solvable. We then give three applications of this tool:
(i) Given $1\leq \ell \leq k-1$, we give a minimum $\ell$-degree condition for which it is polynomial time solvable to determine whether a $k$-graph satisfying this condition has a perfect matching; (ii) Given any graph $F$ we give a minimum degree condition for which it is polynomial time solvable to determine whether a graph satisfying this condition has a perfect $F$-packing;
(iii) We also prove a similar result for perfect $K$-packings in $k$-graphs where
$K$ is a $k$-partite $k$-graph.

For a range of values of $\ell,k$ (i) resolves a conjecture of Keevash, Knox and Mycroft~\cite{KKM13} whilst (ii) answers a question of Yuster~\cite{Yuster07} in the negative.
In many cases our results  are best possible in the sense that lowering the minimum degree condition means that the corresponding decision problem becomes NP-complete.
\end{abstract}
\date{\today}
\maketitle

\section{Introduction}
Given $k\ge 2$, a \emph{$k$-uniform hypergraph} (or \emph{$k$-graph}) consists of a vertex set $V(H)$ and an edge set $E(H)\subseteq \binom{V(H)}{k}$, where every edge is a $k$-element subset of $V(H)$. A \emph{matching} in $H$ is a collection of vertex-disjoint edges of $H$. A \emph{perfect matching} $M$ in $H$ is a matching that covers all vertices of $H$. 

The question of whether a given $k$-graph $H$ contains a perfect matching is one of the most fundamental problems in combinatorics. In the graph case $k=2$, Tutte's Theorem \cite{Tu47} gives necessary and sufficient conditions for $H$ to contain a perfect matching, and Edmonds' Algorithm \cite{Edmonds} finds such a matching in polynomial time. 
On the other hand, the decision problem whether
a $k$-graph contains a perfect matching is famously NP-complete for $k\geq 3$ (see~\cite{karp, garey}). 

An important generalisation of the notion of a perfect matching is that of a \emph{perfect packing}: Given two $k$-graphs $H$  and $F$, an \emph{$F$-packing} in $H$ 
is a collection of vertex-disjoint copies of $F$ in $H$. An
$F$-packing is called \emph{perfect} if it covers all the vertices of $H$.
Perfect $F$-packings are also referred to as \emph{$F$-factors} or \emph{perfect $F$-tilings}. 
Note that  perfect matchings  correspond to the case when $F$ is a single edge. Hell and Kirkpatrick~\cite{HeKi} showed that the decision problem
whether a graph~$G$ has a perfect
$F$-packing is NP-complete precisely when $F$ has a component consisting of at least
$3$ vertices.

In light of the aforementioned complexity results, there has been significant attention to  determine classes 
of (hyper)graphs for which the respective decision problems are polynomial time solvable. 
A key contribution of this paper is to provide a general tool (Theorem~\ref{genthm}) that can be used to obtain such results. For this result we need to introduce several concepts so we defer its statement until Section~\ref{state}.
However, roughly speaking, for any $k$-graph $F$, Theorem~\ref{genthm} yields a general class of $k$-graphs within  which we do have  a complete characterisation of those $k$-graphs  that contain a perfect $F$-packing. We then give three applications of Theorem~\ref{genthm}, which we describe below. In particular, each of our applications convey an underlying theme: In each case, the class of (hyper)graphs $H$ we consider are those that satisfy some minimum degree condition that ensures an \emph{almost}  
  perfect matching or packing $M$ (i.e. $M$ covers all but a constant number of the vertices of  $H$). Thus, in each application we show that we can detect the `last obstructions' to having a perfect matching or packing efficiently.

\subsection{Perfect matchings in hypergraphs}
Given a $k$-graph $H$ with an $\ell$-element vertex set $S$ (where $0 \leq \ell \leq k-1$) we define
$d_H (S)$ to be the number of edges containing $S$. The \emph{minimum $\ell$-degree $\delta _{\ell}
(H)$} of $H$ is the minimum of $d_H (S)$ over all $\ell$-element sets of vertices in $H$.
We refer to $\delta _{k-1}(H)$ as the \emph{minimum codegree} of $H$. The following conjecture from~\cite{HPS, KuOs-survey} gives a minimum $\ell$-degree
condition that ensures a perfect matching in a $k$-graph.

\begin{conjecture}\label{conj}
Let $\ell, k \in \mathbb N$ such that $\ell \leq k-1$.
Given any $\eps >0$, there is an $n_0 \in \mathbb N$ such that the following holds. Suppose $H$ is a $k$-graph on $n\geq n_0$ vertices where $k$ divides $n$.
If
$$\delta _\ell (H) \geq \max \left \{ \left( 1/2 +\eps \right ), \left (1-\left (1-\frac{1}{k}\right)^{k-\ell} +\eps \right ) \right\}\binom{n}{k-\ell}$$
then $H$ contains a perfect matching.
\end{conjecture}
An `exact' version of Conjecture~\ref{conj} (without the error terms) was stated in~\cite{TrZh15}.
There are two types of extremal examples that show, if true, Conjecture~\ref{conj} is asymptotically best possible. 
The first is a so-called \emph{divisibility barrier}: 
Let $V_1$ be a set of $n$ vertices and $A,B$ a partition of $V_1$ where $|A|,|B|$ are as equal as possible whilst ensuring $|A|$ is odd. Let $H_1$ be the $k$-graph with vertex set
$V_1$ and edge set consisting of all those $k$-tuples that contain an even number of vertices from $A$. Then $\delta _{\ell} (H_1)= (1/2+o(1)) \binom{n}{k-\ell}$ for all $1\leq \ell \leq k-1$ but $H_1$ does not contain a perfect matching. (Actually note that there is a \emph{family} of divisibility barrier constructions for this problem; see e.g.~\cite{TrZh15} for more details.)
The second construction is a so-called  \emph{space barrier}: Let $V_2$ be a vertex set of size $n$ and fix $S\subseteq V_2$ with $|S|=n/k-1$. Let $H_2$ be the $k$-graph whose edges are all $k$-sets that intersect $S$. Then $H_2$ does not contain a perfect matching and $\delta _\ell (H_2) =  \left  (1-\left (1-\frac{1}{k}\right)^{k-\ell} +o(1) \right ) \binom{n}{k-\ell}$ for all $1\leq \ell \leq k-1$.

In recent years Conjecture~\ref{conj} (and its exact counterpart) has received substantial attention \cite{AFHRRS, CzKa, HPS, JieNote, Khan1,  Khan2, KO06mat, KOT, MaRu, Pik, RR, RRS06mat, RRS09, TrZh12, TrZh13, TrZh15}. In particular, the \emph{exact} threshold is known for all $\ell$ such that $0.42k \leq \ell \leq k-1$ as well as for a handful of other values of $(k,\ell)$.
For example, R\"odl, Ruci\'nski and Szemer\'edi~\cite{RRS09} determined the codegree threshold for this problem for sufficiently large $k$-graphs $H$ on $n$ vertices. This threshold is $n/2-k+C$ where $C\in\{3/2,2,5/2,3\}$
depends on the value of $n$ and $k$.

Such results give us classes of dense $k$-graphs for which we are certain to have a perfect matching. This raises the question of whether one can lower the minimum $\ell$-degree condition in Conjecture~\ref{conj} whilst still ensuring it is decidable in polynomial time whether such a $k$-graph $H$ has a perfect matching: Let ${\bf PM}(k, \ell, \delta)$ denote the problem of deciding whether there is
a perfect matching in a given $k$-graph on $n$ vertices with minimum $\ell$-degree
at least $\delta \binom{n}{k-\ell}$.  Write ${\bf PM}(k,  \delta):={\bf PM}(k, k-1, \delta)$.

The above mentioned result of R\"odl, Ruci\'nski and Szemer\'edi~\cite{RRS09} implies that ${\bf PM}(k,  1/2)$ is in P. 
On the other hand, for $k \geq 3$ Szyma\'nska~\cite{Szy13} proved that for $\delta<1/k$ the problem $\PM(k,\delta)$ admits a polynomial-time reduction to $\PM(k,0)$ and hence $\PM(k,\delta)$ is also NP-complete. 
Karpi\'nski, Ruci\'nski and Szyma\'nska~\cite{KRS10} proved that there exists an $\e>0$ such that $\PM(k,1/2-\e)$ is in P; they also raised the question of determining the complexity of $\PM(k,\delta)$ for $\delta\in [1/k, 1/2)$. For any $\delta >1/k$, Keevash, Knox and Mycroft~\cite{KKM13} recently proved that $\PM(k,\delta)$ is in $P$. Then very recently this question was completely resolved by the first author~\cite{Han14_poly}
who showed that $\PM(k,\delta)$ is in $P$ for any $\delta \geq 1/k$. 

Note that the minimum codegree of the space barrier construction $H_2$ above is $\delta _{k-1}(H_2)=n/k-1$. So in the case of minimum codegree, the threshold at which $\PM(k,\delta)$ `switches' from NP-complete to P corresponds
to this space barrier. This leads to the question whether the same phenomenon occurs in the case of minimum $\ell$-degree for $\ell \leq k-2$. 
In support of this, Szyma\'nska~\cite{Szy13}  proved that ${\bf PM}(k, \ell, \delta)$ is NP-complete when
$\delta <1-(1-1/k)^{k-\ell}$. This led Keevash, Knox and Mycroft~\cite{KKM13} to pose the following conjecture.
\begin{conjecture}[Keevash, Knox and Mycroft~\cite{KKM13}]\label{conj1}
${\bf PM}(k, \ell, \delta)$ is in P for every $\delta >1-(1-1/k)^{k-\ell}$.
\end{conjecture}
As an application of Theorem~\ref{genthm} we verify Conjecture~\ref{conj1} in a range of cases. To state our result, we first must introduce the notion of a perfect fractional matching:
Let $H$ be a $k$-graph on $n$ vertices. A \emph{fractional matching} in $H$ is a function $w: E(H) \rightarrow [0,1]$ such that for each $v \in V(H)$ we have that $\sum _{e \ni v} w(e)\leq1$. Then $\sum _{e \in E(H)} w(e)$ is the \emph{size} of $w$. If the size of the largest fractional matching $w$ in $H$ is $n/k$ then we say that $w$ is a \emph{perfect fractional matching}.
Given $k,\ell \in \mathbb N$ such that $\ell \leq k-1$, define $c^*_{k,\ell}$ to be the smallest number $c$ such that every $k$-graph $H$ on $n$ vertices with $\delta _{\ell} (H) \geq (c +o(1)) \binom{n-\ell}{k-\ell}$ contains a perfect fractional matching.
We can now state our complexity result for perfect matchings.
\begin{theorem}\label{mainthm}
Given  $k,\ell \in \mathbb N$ such that $1\le \ell \leq k-1$, define $\delta ^*:= \max \{ 1/3, c^* _{k,\ell}\}$. Given any $\delta \in (\delta ^*,1]$, ${\bf PM}(k, \ell, \delta)$ is in $P$. 
That is, for every $n$-vertex $k$-graph $H$ with minimum $\ell$-degree at least $\delta \binom{n-\ell}{k-\ell}$, there is an algorithm with running time $O(n^{k^2})$ which determines whether $H$ contains
a perfect matching.
\end{theorem}
Alon,  Frankl,  Huang, R\"odl, Ruci\'nski, and  Sudakov~\cite{AFHRRS} conjectured that $c^*_{k,\ell}=1-(1-1/k)^{k-\ell}$ for all $\ell,k \in \mathbb N$. Thus, Theorem~\ref{mainthm}
verifies Conjecture~\ref{conj1} in all cases where $c^*_{k,\ell}=1-(1-1/k)^{k-\ell}$ and $c^*_{k,\ell}\geq 1/3$. 
In particular, K\"uhn, Osthus and Townsend~\cite[Theorem 1.7]{KOTo} proved that $c^*_{k,\ell}=1-(1-1/k)^{k-\ell}$ in the case when $\ell \geq k/2$ and the first author~\cite[Theorem 1.5]{JieNote} proved that $c^*_{k,\ell}=1-(1-1/k)^{k-\ell}$ in the case when $\ell = (k-1)/2$.

Note that for all $1\leq \ell \leq k-1$,
\[
\left( \frac{k-1}{k} \right) ^{k-\ell} < \left ( \frac{1}{e} \right)^{1-\frac{\ell}{k}} .
\]
Thus,  $1-(1-1/k)^{k-\ell}\geq 1/3$ if $\ell \leq (1+\ln (2/3))k \approx 0.5945k$. (Here $\ln$ denotes the natural logarithm function.) Altogether, this implies the following.
\begin{corollary}
 Conjecture~\ref{conj1} holds for all $k,\ell \in \mathbb N$ such that
$(k-1)/2 \leq \ell \leq (1+\ln (2/3))k$.
\end{corollary}

\subsection{Perfect packings in graphs}
Several complexity problems for perfect packings in graphs have received attention. Given a graph $F$, we write $|F|$ for its order and $\chi(F)$ for its chromatic number.
For approximating the size of a maximal $F$-packing, Hurkens and Schrijver \cite{HuSc} gave an $(|F|/2+\e)$-approximation algorithm (where $\e>0$ is arbitrary) which runs in polynomial time.
On the other hand, Kann \cite{Kann} proved that the problem is APX-hard if $F$ has a component which contains at least three vertices.
(In other words, it is impossible to approximate the optimum solution within an arbitrary factor unless P$=$NP.)
In contrast, the results in \cite{HeKi} imply that the remaining cases of the problem can be solved in polynomial time.

The following classical result of Hajnal and Szemer\'edi~\cite{hs} characterises the minimum degree that ensures a graph contains a perfect $K_r$-packing. 
\begin{theorem}[Hajnal and Szemer\'edi~\cite{hs}]\label{hs}
Every graph $G$ whose order $n$
is divisible by $r$ and whose minimum degree satisfies $\delta (G) \geq (1-1/r)n$ contains a perfect $K_r$-packing. 
\end{theorem}
By considering a complete $r$-partite graph $G$ with vertex classes of almost equal size, one can see that the minimum degree condition in Theorem~\ref{hs} cannot be lowered.
Kierstead, Kostochka, Mydlarz and Szemer\'edi~\cite{KKMS} gave a version of Theorem~\ref{hs} which also yields a fast (polynomial time) algorithm for producing the perfect $K_r$-packing.

Up to an error term, the following theorem of Alon and Yuster \cite{AY96} generalises Theorem~\ref{hs}.
Let $M(n)$ be the time needed to multiply two $n$ by $n$ matrices with $0,1$ entries. (Here the entries are viewed as elements of $\mathbb Z$.)
Determining $M(n)$ is a challenging problem in theoretic computer science, and the best known bound of  $M(n)=O(n^{2.3728639})$ was obtained by Le Gall~\cite{LeGall14}.

\begin{theorem}[Alon and Yuster \cite{AY96}]\label{thm:AY}
For every $\r>0$ and each graph $F$ there exists an integer $n_0=n_0(\r, F)$ such that every graph $G$ whose order $n\ge n_0$ is divisible by $|F|$ and whose minimum degree is at least $(1-1/\chi(F)+\r)n$ contains a perfect $F$-packing.
Moreover, there is an algorithm which finds this $F$-packing in time $O(M(n))$.
\end{theorem}

In~\cite{AY96}, they also conjectured that the error term $\r n$ in Theorem~\ref{thm:AY} can be replaced by a constant $C(F)>0$ depending only on $F$; this has been verified by Koml\'{o}s, S\'{a}rk\"{o}zy and Szemer\'{e}di \cite{KSS-AY}.

\begin{theorem}[Koml\'{o}s, S\'{a}rk\"{o}zy and Szemer\'{e}di \cite{KSS-AY}]\label{thm:KSS}
For every graph $F$ there exist integers $C <|F|$ and $n_0=n_0(F)$ such that every graph $G$ whose order $n\ge n_0$ is divisible by $|F|$ and whose minimum degree is at least $(1-1/\chi(F))n+C$ contains a perfect $F$-packing.
Moreover, there is an algorithm which finds this $F$-packing in time $O(n M(n))$.
\end{theorem}

As observed in \cite{AY96}, there are graphs $F$ for which the constant $C(F)$ cannot be omitted completely.
On the other hand, there are graphs $F$ for which the minimum degree condition in Theorem~\ref{thm:KSS} can be improved significantly \cite{Kawa, CKO}, by replacing the chromatic number with the critical chromatic number.
The \emph{critical chromatic number} $\chi_{cr}(F)$ of a graph $F$ is defined as $(\chi(F)-1)|F|/(|F|-\sigma(F))$, where $\sigma(F)$ denotes the minimum size of the smallest colour class in a colouring of $F$ with $\chi(F)$ colours.
Note that $\chi(F)-1<\chi_{cr}(F)\le\chi(F)$ and the equality holds if and only if every $\chi(F)$-colouring of $F$ has equal colour class sizes.
If $\chi_{cr}(F)=\chi(F)$, then we call $F$ \emph{balanced}, otherwise \emph{unbalanced}.
Koml\'{o}s \cite{Komlos} proved that one can replace $\chi(F)$ with $\chi_{cr}(F)$ in Theorem~\ref{thm:KSS} at the price of obtaining an $F$-packing covering all but $\e n$ vertices.
He also conjectured that the error term $\e n$ can be replaced with a constant that only depends on $F$ \cite{Komlos}; this was confirmed by Shokoufandeh and Zhao \cite{ShZh} (here we state their result in a slightly weaker form).

\begin{theorem}[Shokoufandeh and Zhao \cite{ShZh}]\label{thm:ShZh}
For any $F$ there is an $n_0=n_0(F)$ so that if $G$ is a graph on $n\ge n_0$ vertices and minimum degree at least $(1-1/\chi_{cr}(F))n$, then $G$ contains an $F$-packing that covers all but at most $5|F|^2$ vertices.
\end{theorem}

Then the question is, \emph{for which $F$ can we replace $\chi(F)$ with $\chi_{cr}(F)$ in Theorem~\ref{thm:KSS}?}
K\"{u}hn and Osthus \cite{KuOs06soda, KuOs09} answered this question completely.
To state their result, we need some definitions.
Write $k:=\chi(F)$.
Given a $k$-colouring $c$, let $x_1\le \cdots \le x_k$ denote the sizes of the colour classes of $c$ and put $D(c)=\{x_{i+1}-x_i \mid i\in [k-1]\}$.
Let $D(F)$ be the union of all the sets $D(c)$ taken over all $k$-colourings $c$.
Denote by $\hcf_{\chi}(F)$ the highest common factor of all integers in $D(F)$. (If $D(F)=\{0\}$, then set $\hcf_{\chi}(F):=\infty$.)
Write $\hcf_c(F)$ for the highest common factor of all the orders of components of $F$ (for example $\hcf_c(F)=|F|$ if $F$ is connected).
If $\chi(F)\neq 2$, then define $\hcf(F)=1$ if $\hcf_{\chi}(F)=1$.
If $\chi(F)= 2$, then define $\hcf(F)=1$ if both $\hcf_{c}(F)=1$ and $\hcf_{\chi}(F)\le 2$.
Then let
\[
\chi_*(F) = \begin{cases}
\chi_{cr}(F) &\text{ if } \hcf(F)=1,\\
\chi(F) &\text{ otherwise.}
\end{cases}
\]
In particular we have $\chi_{cr}(F)\le \chi_*(F)$.

\begin{theorem}[K\"{u}hn and Osthus \cite{KuOs06soda, KuOs09}]\label{thm:KO}
There exist integers $C=C(F)$ and $n_0=n_0(F)$ such that every graph $G$ whose order $n\ge n_0$ is divisible by $|F|$ and whose minimum degree is at least $(1-1/\chi_{*}(F))n+C$ contains a perfect $F$-packing.
\end{theorem}

Theorem~\ref{thm:KO} is best possible in the sense that the degree condition cannot be lowered up to the constant $C$ (there are also graphs $F$ such that the constant cannot be omitted entirely).
Moreover, this also implies that, \emph{one can replace $\chi(F)$ with $\chi_{cr}(F)$ in Theorem~\ref{thm:KSS} if and only if $\hcf(F)=1$}.
When $\hcf(F)\not =1$ certain \emph{divisibility barrier} constructions show that the minimum degree condition in Theorem~\ref{thm:KO} (and thus Theorem~\ref{thm:KSS}) is best possible up to the additive constant $C$ (see \cite{KuOs09}). On the other hand, the following \emph{space barrier} construction shows that one cannot replace $\chi_{*}(F)$ with anything smaller than $\chi_{cr}(F)$  in Theorem~\ref{thm:KO}; that is, when $\hcf(F)\not =1$,
Theorem~\ref{thm:KO}  is best possible up to the additive constant $C$: Let $G$ be the complete $\chi(F)$-partite graph on $n$ vertices with $\sigma(F)n/|F|-1$ vertices in one vertex class, and the other vertex classes of sizes as equal as possible. Then $\delta (G) = (1-1/\chi_{cr}(F))n-1$ and $G$ does not contain a perfect $F$-packing.

Now let us return to the algorithmic aspect of this problem.
Let $\pack(F, \delta)$ be the decision problem of determining whether a graph $G$ whose minimum degree is at least $\delta |G|$ contains a perfect $F$-packing.
When $F$ contains a component of size at least $3$, the result of Hell and Kirkpatrick \cite{HeKi} shows that $\pack(F,0)$ is NP-complete.
In contrast, Theorem~\ref{thm:KO} gives that $\pack(F, \delta)$ is (trivially) in P for any $\delta\in (1-1/\chi_{*}(F),1]$.
In~\cite{KuOs06soda}, K\"{u}hn and Osthus showed that $\pack(F, \delta)$ is NP-complete for any $\delta\in [0, 1-1/\chi_{cr}(F))$ if $F$ is a clique of size at least $3$ or a complete $k$-partite graph such that $k\ge 2$ and the size of the second smallest vertex class is at least $2$.

Due to lack of knowledge on the range $\delta\in [0, 1-1/\chi_{*}(F))$ for general $F$, we still do not understand $\pack(F, \delta)$ well in general.
Indeed, even for (unbalanced) complete multi-partite graphs $F$ with $\hcf(F)\neq 1$, there is a substantial hardness gap for $\delta\in [1-1/\chi_{cr}(F), 1-1/\chi_{*}(F)]$.
In particular, Yuster asked the following question in his survey \cite{Yuster07}.

\begin{problem}[Yuster \cite{Yuster07}]\label{prob}
Is it true that $\mathbf{Pack}(F, \delta)$ is NP-complete for all $\delta\in [0, 1-1/\chi_{*}(F))$ and any $F$ which contains a component of size at least $3$?
\end{problem}

Our next result provides an algorithm showing that $\mathbf{Pack}(F, \delta)$ is in P when $\delta\in (1-1/\chi_{cr}(F),1]$, which  gives a negative answer to Problem~\ref{prob} (as seen for any $F$ such that $\chi_{cr}(F)<\chi_*(F)$).
In fact, this gives the first \emph{nontrivial} polynomial-time algorithm for the decision problem $\pack(F, \delta)$.
In particular, it eliminates the aforementioned hardness gap for unbalanced complete multi-partite graphs $F$ with $\hcf(F)\neq 1$ almost entirely.

\begin{theorem}\label{thm:Ftil}
For any $m$-vertex $k$-chromatic graph $F$ and $\delta\in (1-1/\chi_{cr}(F),1]$, $\mathbf{Pack}(F, \delta)$ is in $P$. 
That is, for every $n$-vertex graph $G$ with minimum degree at least $\delta n$, there is an algorithm with running time $O(n^{\max\{2^{m^{k-1}-1}m+1, \,m(2m-1)^m\}})$, which determines whether $G$ contains a perfect $F$-packing.
\end{theorem}

In view of the aforementioned result of \cite{KuOs06soda}, Theorem~\ref{thm:Ftil} is asymptotically best possible if $F$ is a complete $k$-partite graph such that $k\ge 2$ and the size of the second smallest cluster is at least $2$ (note that when $F$ is balanced, the result is included in Theorem~\ref{thm:AY}).
On the other hand, Theorem~\ref{thm:Ftil} complements Theorem~\ref{thm:ShZh} in the sense that when the minimum degree condition guarantees an $F$-packing that covers all but constant number of vertices, we can detect the `last obstructions' efficiently.

We remark that Theorem~\ref{thm:Ftil} also appears in a conference paper of the first author~\cite{jieconf}.

\subsection{Perfect packings in hypergraphs} Over the last few years there has been an interest in obtaining degree conditions that force a perfect $F$-packing in $k$-graphs where $k \geq 3$. 
In general though, this appears to be a harder problem than the graph version. Indeed, far less is known in the hypergraph case. See a survey of Zhao~\cite{zsurvey} for an overview of the known results in the area.
Our final application of Theorem~\ref{genthm} is related to a recent general result of Mycroft~\cite{Mycroft}.

Given a $k$-graph $F$ and an integer $n$ divisible by $|F|$, we define the threshold ${\delta(n,F)}$ 
 as the smallest integer $t$ such that every $n$-vertex $k$-graph $H$ with $\delta_{k-1}(H)\ge t$ contains a perfect $F$-packing.
Let $F$ be a $k$-partite $k$-graph on vertex set $U$ with at least one edge.
Then a \emph{$k$-partite realisation} of $F$ is a partition of $U$ into vertex classes $U_1,\dots, U_k$ so that for any $e\in E( F)$ and $1\le j\le k$ we have $|e\cap U_j|=1$. 
Define
\[
\mathcal{S}(F):= \bigcup_{\chi} \{|U_1|,\dots, |U_k|\} \text{ and } \mathcal{D}(F):=\bigcup_{\chi} \{||U_i| - |U_j||: i, j\in [k]\},
\]
where in each case the union is taken over all $k$-partite realisations $\chi$ of $F$ into vertex classes $U_1,\dots, U_k$ of $F$.
Then $\gcd(F)$ is defined to be the greatest common divisor of the set $\mathcal{D}(F)$ (if $\mathcal{D}(F)=\{0\}$ then $\gcd(F)$ is undefined).
We also define
\[
\sigma(F):= \frac{\min_{S\in \mathcal{S}(F)}S}{|V(F)|},
\]
and thus in particular, $\sigma(F)\le 1/k$.
Mycroft~\cite{Mycroft} proved the following:
\begin{equation} \label{eq:nk1}
\delta(n, F) \le \left\{\begin{array}{ll}
{n}/{2} +o(n) & \text{if } \mathcal{S}(F)=\{1\} \text{ or } \gcd(\mathcal{S}(F))>1;\\
\sigma(F)n +o(n)  & \text{if } \gcd(F)=1;\\
\max \{\sigma(F)n, n/p\}+o(n) & \text{if } \gcd(\mathcal{S}(F))=1 \text{ and } \gcd(F)=d>1, 
\end{array} \right.\end{equation} 
where $p$ is the smallest prime factor of $d$.
Moreover, equality holds in~\eqref{eq:nk1} for all complete $k$-partite $k$-graphs $F$, as well as a wide class of other $k$-partite $k$-graphs.

Mycroft~\cite{Mycroft} also showed that minimum codegree of at least $\sigma(F)n +o(n)$ in an $n$-vertex $k$-graph $H$  ensures an  $F$-packing covering all but a constant number of vertices.
The next two results show that above this degree threshold, one can determine in polynomial time whether $H$ contains a perfect $F$-packing, whilst below the threshold the problem is NP-complete (for complete $k$-partite $k$-graphs $F$).
Given $\delta >0$ and a $k$-graph $F$, let $\pack(F, \delta)$ be the decision problem of determining whether a $k$-graph $H$ whose minimum codegree is at least $\delta |H|$ contains a perfect $F$-packing.

\begin{theorem}\label{prop:reduc2}
Let $k\ge 3$ be an integer and let $F$ be a complete $k$-partite $k$-graph.
Then $\mathbf{Pack}(F, \delta)$ is NP-complete for any $\delta\in [0, \sigma (F))$. 
\end{theorem}


\begin{theorem}\label{thm:Ktil}
Let $k\ge 3$ be an integer and let $F$ be an $m$-vertex $k$-partite $k$-graph.
For any $\delta\in (\sigma (F),1]$, $\mathbf{Pack}(F, \delta)$ is in $P$. 
That is, for every $n$-vertex $k$-graph $H$ with $\delta_{k-1}(H)\ge \delta n$, there is an algorithm with running time $O(n^{m(2m-1)^m})$, which determines whether $H$ contains a perfect $F$-packing.
\end{theorem}
Note that when $F$ is just an edge, a perfect $F$-packing is simply a perfect matching. Further, in this case $\sigma (F)=1/k$.
Thus, Theorem~\ref{thm:Ktil} is a generalisation of the perfect matching result of Keevash, Knox and Mycroft~\cite{KKM13}.

\subsection{A general tool for complexity results} To prove the results mentioned above, we introduce a general structural theorem, Theorem~\ref{genthm}. 
Given any $k$-graph $F$, Theorem~\ref{genthm} considers $k$-graphs $H$ whose minimum $\ell$-degree is sufficiently large so as to ensure $H$ contains an \emph{almost} perfect $F$-packing (that is an $F$-packing covering all but a constant number of vertices in $H$). To state Theorem~\ref{genthm} we introduce a coset group which, loosely speaking, is defined with respect to the `distribution' of copies of $F$ in $H$.
In particular, Theorem~\ref{genthm} states that if this coset group $Q$ has bounded size then we have a necessary and sufficient condition for $H$ containing a perfect $F$-packing. 
This condition can be easily checked in polynomial time. This means if we have a class of $k$-graphs $H$ (i) each of whose minimum $\ell$-degree is sufficiently large and; (ii) each such $H$ has a corresponding coset group $Q$ of bounded size, then we can determine in polynomial time whether an element $H$ in this class has a perfect $F$-packing.

Thus, in applications of Theorem~\ref{genthm} the key goal is to determine whether the corresponding coset groups have bounded size. In our applications to Theorems~\ref{thm:Ftil} and~\ref{thm:Ktil} all $k$-graphs $H$ considered will have a corresponding coset group $Q$  of bounded size. On the other hand, to prove Theorem~\ref{mainthm} we show that a hypergraph $H$ under consideration must have  a corresponding coset group $Q$  of bounded size, or failing that, \emph{must} have a perfect matching. 

The approach of using these auxiliary coset groups as a tool for such complexity results was also used in~\cite{KKM13, Han14_poly}; note that these applications were for perfect matchings in hypergraphs of large minimum codegree. Theorem~\ref{genthm} provides a generalisation of this approach. Indeed, Theorem~\ref{genthm} is applicable to perfect matching \emph{and} packing problems in  (hyper)graphs of large minimum $\ell$-degree for \emph{any} $\ell$. As such, we suspect Theorem~\ref{genthm} could have many more applications in the area.


The paper is organised as follows. In the next section we prove Theorem~\ref{prop:reduc2}. In Section~\ref{sec3} we introduce the general structural theorem (Theorem~\ref{genthm}) as well as some notation and definitions.
We prove Theorem~\ref{genthm} in Sections~\ref{secabs} and~\ref{secgen}. In Sections~\ref{sec6} and~\ref{sec7} we introduce some tools that are useful for the applications of Theorem~\ref{genthm}.
We then prove Theorems~\ref{mainthm},~\ref{thm:Ftil} and~\ref{thm:Ktil} in Sections~\ref{sec8}, \ref{sec9} and~\ref{sec10} respectively.


\section{Proof of the  hardness result}


In this section we prove Theorem~\ref{prop:reduc2}.

\begin{proof}[Proof of Theorem~\ref{prop:reduc2}]
Our proof resembles the one of Szyma\'nska~\cite[Theorem 1.7]{Szy13} and we also use the following result from it.
Let $\PM_{lin}(k)$ be the subproblem of $\PM(k, 0)$ restricted to $k$-uniform hypergraphs which are \emph{linear}, that is, any two edges share at most one vertex.
Then it is shown in \cite{Szy13} that $\PM_{lin}(3)$ is NP-complete.

Let $K:=K^{(k)}(a_1, \dots, a_k)$ be the complete $k$-partite $k$-graph of order $m$ with vertex classes of size $a_1\le \cdots\le a_k$.
We may assume that $a_k\ge 2$ as otherwise $K$ is just a single edge and $\pack(K, \delta)$ is NP-complete for $\delta\in [0,1/k)$ as shown in \cite{Szy13}.
We prove the theorem by the following reductions.
\[
\PM_{lin}(3)\overset{(a)}{\le} \PM_{lin}(m) \overset{(b)}{\le} \pack(K, 0)\overset{(c)}{\le} \pack(K, \delta).
\]

\noindent\emph{Reduction} (a).
In fact, we will show that $\PM_{lin}(k){\le} \PM_{lin}(k+1)$ for any $k\ge 3$.
Let $H$ be a linear $k$-graph with $n$ vertices and $s$ edges.
We construct a linear $(k+1)$-graph $G$ by taking $k+1$ disjoint copies $H_i$ of $H$, $i\in [k+1]$ and for every edge $e$ in each copy $H_i$ we add one vertex $v_i^e$ to $V(G)$, i.e., $V(G)=\bigcup_{i\in [k+1]} (V(H_i)\cup \bigcup_{e\in E(H_i)}\{v_i^e\})$.
Thus $|V(G)|=(k+1)(n+s)$.
For every $e\in E(H)$ the $(k+1)$-tuple $\{v_i^e: i\in[k+1]\}$ forms an edge of $G$.
Moreover, we add to $E(G)$ all sets of the form $e\cup \{v_i^e\}$ for all $i\in [k+1]$ and $e\in E(H_i)$.
Hence, $G$ has $(k+2)s$ edges and is linear by the definition.

Suppose $H$ has a perfect matching $M$.
Let $M_i$ be the same matching in the copy $H_i$ of $H$, $i\in [k+1]$.
Then it is easy to see that $G$ has a perfect matching $M'=\{e\cup \{v_i^e\}, e\in M_i, i\in [k+1]\}\cup \{f_e=\{v_1^e,\dots, v_{k+1}^e\}: e\notin M\}$.
On the other hand assume that $G$ has a perfect matching $M'=\{f_1,\dots, f_{n+s}\}$.
For all $v\in V(H_1)$, let $f(v)$ be such that $f(v)\in M'$ and $v\in f(v)$.
But the only edges of $G$ containing the vertices of $H_1$ are of the form $e\cup \{v_1^e\}$, so $|\{f(v):v\in V(H_1)\}|=n/k$ and $\{f(v)\cap V(H_1): v\in V(H_1)\}$ is a perfect matching of $H_1$.
Therefore $H$ also has a perfect matching.

\medskip
\noindent\emph{Reduction} (b).
Given a linear $m$-graph $H$ we build a $k$-graph $G$ by replacing each edge of $H$ with a copy of $K$.
If $H$ has a perfect matching then $G$ has a perfect $K$-packing.
In turn, if $G$ has a perfect $K$-packing, then by the linearity of $H$, each copy of $K$ corresponds to a single edge of $H$ and therefore the $K$-packing corresponds to a perfect matching of $H$.
In fact, since $K$ is complete $k$-partite, there exists an ordering $e_1,\dots, e_t$ of $E(K)$ (e.g., the lexicographic ordering) such that for any $2\le i\le t$, there exists $1\le j\le i-1$ such that $|e_i\cap e_{j}|\ge 2$.
Then by the linearity of $H$, each copy of $K$ corresponds to a single edge of $H$.

\medskip
\noindent\emph{Reduction} (c).
Let $\r:=\sigma(K) - \delta=a_1/m - \delta$ and thus $\r>0$.
To achieve this, for each instance $H$ of $\pack(K, 0)$ with $n$ vertices such that $m\mid n$, we define a graph $H'$ as follows. 
Let $H_0=H_0(k, n, \r)$ be a $k$-graph, in which the vertex set is the union of two disjoint sets $A\cup B$, such that $|A|=a_1\lceil n/\r\rceil$ and $|B|=(m-a_1)\lceil n/\r\rceil$.
The edge set of $H_0$ consists of all $k$-vertex sets of $A\cup B$ which have a non-empty intersection with $A$.
Observe that $\delta_{k-1}(H_0)=|A|$ and $H_0$ has a perfect $K$-packing (in which each copy of $K$ contains $a_1$ vertices in $A$ and $m-a_1$ vertices in $B$).
Then let $H'$ be the $k$-graph such that $V(H')=V(H)\cup V(H_0)$ and $E(H')=E(H)\cup E$, where $E$ consists of all $k$-sets that intersect $A$ and thus $E(H_0)\subseteq E$.
Clearly $|V(H')|=n+ m\lceil n/\r \rceil$ and
\[
\delta_{k-1}(H') = |A|= a_1\lceil n/\r \rceil \ge \left(\frac{a_1}{m} - \frac{a_1}{m^2} \r \right) |V(H')| > \delta |V(H')|.
\]
If $H$ has a perfect $K$-packing, so does $H'$.
Now suppose that $H$ does not have a perfect $K$-packing and $H'$ has a perfect $K$-packing $M$.
This means that there exists a copy of $K$ in $M$ with its vertex set denoted by $K'$, such that $K'\cap A\neq \emptyset$ and $K'\cap V(H)\neq \emptyset$.
First assume that $K'\cap B=\emptyset$.
Then since $\frac{|A\setminus K'|}{|B\setminus K'|} = \frac{|A\setminus K'|}{|B|} < a_1/(m-a_1)$, the vertices of $B$ cannot be covered completely by $M$, contradicting the existence of $M$.
Otherwise $K'\cap B\neq\emptyset$. Then clearly $1\le |K'\cap B|\le m-a_1-1$ and $|A\cap K'|\ge a_1$.
Again, $\frac{|A\setminus K'|}{|B\setminus K'|} \le  \frac{|A|-a_1}{|B|-(m-a_1-1)} < a_1/(m-a_1)$, so the rest of the vertices of $B$ cannot be covered completely by $M$, a contradiction.
\end{proof}

\section{The general structural theorem}\label{sec3}
In order to state our general structural theorem, Theorem~\ref{genthm}, we will now introduce some definitions and notation.
\subsection{Almost perfect packings}
Let $k,\ell \in \mathbb N$ where $\ell\leq k-1$. Let $F$ be an $m$-vertex $k$-graph and $D \in \mathbb N$.
Define $\delta(F, \ell, D)$
to be the smallest number $\d$ such that every $k$-graph $H$ on $n$ vertices with $\delta _{\ell} (H) \geq (\d +o(1)) \binom{n-\ell}{k-\ell}$ contains an $F$-packing covering all but at most $D$ vertices.
We write $\delta(k, \ell, D)$ for $\delta(F, \ell, D)$ when $F$ is a single edge. 


\subsection{Lattices and solubility}
One  concept needed to understand the statement and proof of Theorem~\ref{genthm} is that of \emph{lattices and solubility} introduced by Keevash, Knox and Mycroft~\cite{KKM13}.
Let $H$ be an $n$-vertex $k$-graph.
We will work with a vertex partition $\cP=\{ V_1, \dots, V_d\}$ of $V(H)$ for some integer $d\ge 1$.
In this paper, every partition has an implicit ordering of its parts.
The \emph{index vector} $\bfi_{\cP}(S)\in \mathbb{Z}^d$ of a subset $S\subseteq V(H)$ with respect to $\cP$ is the vector whose coordinates are the sizes of the intersections of $S$ with each part of $\cP$, namely, $\bfi_{\cP}(S)|_i=|S\cap V_i|$ for $i\in [d]$, where $\bfv|_i$ is defined as the $i$th digit of $\bfv$. 
For any $\bfv=\{v_1,\dots, v_d\}\in \mathbb{Z}^d$, let $|\bfv|:=\sum_{i=1}^d v_i$.
We say that $\bfv\in \mathbb{Z}^d$ is an \emph{$r$-vector} if it has non-negative coordinates and $|\bfv|=r$. 

Let $F$ be an $m$-vertex $k$-graph and let $\mu>0$.
Define $I_{\cP, F}^\mu(H)$ to be the set of all $\bfi\in \mathbb{Z}^d$ such that $H$ contains at least $\mu n^m$ copies of $F$ with index vector $\bfi$ and let $L_{\cP,F}^{\mu}(H)$ denote the lattice in $\mathbb{Z}^d$ generated by $I_{\cP,F}^\mu(H)$.

Let $q \in \mathbb N$.
A (possibly empty) $F$-packing $M$ in $H$ of size at most $q$ is a \emph{$q$-solution} for $(\cP, L_{\cP,F}^{\mu}(H))$ (in $H$) if $\bfi_{\cP}(V(H)\setminus V(M))\in L_{\cP,F}^{\mu}(H)$; we say that $(\cP, L_{\cP,F}^{\mu}(H))$ is \emph{$q$-soluble} if it has a $q$-solution.

Given a partition $\cP$ of $d$ parts, we write $L_{\max}^d$ for the lattice generated by all $m$-vectors. So $L_{\max}^d:=\{ \bfv\in \mathbb{Z}^d: m \text{ divides }|\bfv| \}$.

Suppose $L\subset L_{\max}^{|\cP|}$ is a lattice in $\mathbb{Z}^{|\cP|}$, where $\cP$ is a partition of a set $V$.
The \emph{coset group} of $(\cP, L)$ is $Q=Q(\cP, L):=L_{\max}^{|\cP|}/L$.
For any $\bfi\in L_{\max}^{|\cP|}$, the \emph{residue} of $\bfi$ in $Q$ is $R_Q(\bfi):=\bfi+L$. For any $A\subseteq V$ of size divisible by $m$, the \emph{residue} of $A$ in $Q$ is $R_Q(A):=R_Q(\bfi_{\cP}(A))$.

\subsection{Reachability and good partitions}\label{sub23}
Let $F$ be an $m$-vertex $k$-graph and let $H$ be an $n$-vertex $k$-graph.
We say that two vertices $u$ and $v$ in $V(H)$ are \emph{$(F, \beta, i)$-reachable in $H$} if there are at least $\beta n^{i m-1}$ $(i m-1)$-sets $S$ such that both $H[S\cup \{u\}]$ and $H[S\cup \{v\}]$ have perfect $F$-packings. 
We refer to such a set $S$ as a \emph{reachable $(im-1)$-set  for $u$ and $v$}.
We say a vertex set $U\subseteq V(H)$ is \emph{$(F, \beta, i)$-closed in $H$} if any two vertices $u,v\in U$ are $(F, \beta, i)$-reachable in $H$.
Given any $v \in V(H)$, define $\tilde{N}_{F,\b, i}(v,H)$ to be the set of vertices in $V(H)$ that are $(F,\beta,i)$-reachable to $v$ in $H$.

Let $\b,c >0$ and $t \in \mathbb N$. A partition $\mathcal P=\{V_1, \dots, V_d\}$ of $V(H)$ is \emph{$(F,\b,t,c)$-good} if the following properties hold:
\begin{itemize}
\item $V_i$ is $(F,\b,t)$-closed in $H$ for all $i \in [d]$;
\item $|V_i|\geq cn$ for all $i \in [d]$.
\end{itemize}

\subsection{Statement of the general structural theorem}\label{state}
With these definitions to hand, we are now able to state the general structural theorem.
Throughout the paper, we write $0<\alpha \ll \beta \ll \gamma$ to mean that we can choose the constants
$\alpha, \beta, \gamma$ from right to left. More
precisely, there are increasing functions $f$ and $g$ such that, given
$\gamma$, whenever we choose some $\beta \leq f(\gamma)$ and $\alpha \leq g(\beta)$, all
calculations needed in our proof are valid. 
Hierarchies of other lengths are defined in the obvious way.

\begin{theorem}[Structural Theorem]\label{genthm}
Let $k,\ell \in \mathbb N$ where $\ell\leq k-1$
and let $F$ be an $m$-vertex $k$-graph. Define $D,q,t,n_0 \in \mathbb N$ and $\b,\mu, \gamma , c>0$ where
$$1/n_0 \ll \b, \mu \ll \gamma, c , 1/m, 1/D, 1/q, 1/t.$$
Let $H$ be a $k$-graph on $n \geq n_0$ vertices where $m$ divides $n$. Suppose that 
\begin{itemize}
\item[(i)] $\delta _{\ell} (H) \geq (\delta(F, \ell, D)+\gamma)\binom{n-\ell}{k-\ell}$;
\item[(ii)] $\mathcal P=\{V_1, \dots, V_d\}$ is an $(F,\b,t,c)$-good partition of $V(H)$;
\item[(iii)] $|Q(\mathcal P, L^{\mu}_{\mathcal P,F}(H))|\leq q$.
\end{itemize}
Then $H$ contains a perfect $F$-packing if and only if $(\mathcal P, L^{\mu}_{\mathcal P,F}(H))$ is $q$-soluble.
\end{theorem}
At first sight Theorem~\ref{genthm} may seem somewhat technical.
In particular, it may not be clear the roles that conditions (i)--(iii) play. We will explain this in more detail now.

In the proof of (the backward implication of) Theorem~\ref{genthm} we will utilise the \emph{absorbing method}. This technique was initiated  by R\"odl, Ruci\'nski and Szemer\'edi \cite{RRS06} and has proven to be a powerful tool for finding spanning structures in graphs and hypergraphs. 
 Fix an integer $i>0$ and a $k$-graph $F$. Let $H$ be a $k$-graph. For a set $S\subseteq V(H)$, we say a set $T\subseteq V(H)$ is an \emph{absorbing $(F,i)$-set for $S$} if $|T|=i$ and both $H[T]$ and $H[T\cup S]$ contain perfect $F$-packings. Informally, we will refer to $T$ as an \emph{absorbing set for $S$} and say $T$ \emph{absorbs} $S$.

Often in  proofs employing the absorbing method the goal is to find some small set $A$ such that for \emph{any} very small set of vertices $S$ in $H$, $A$ absorbs $S$.
In particular, if one could guarantee such a set $A$ in Theorem~\ref{genthm} then we would ensure a perfect $F$-packing:
By (i), $H\setminus A$ would have an almost perfect $F$-packing. Then $A$ can be used to absorb the uncovered vertices to obtain a perfect $F$-packing.

Not all $k$-graphs satisfying the hypothesis of Theorem~\ref{genthm} will have a perfect $F$-packing; so one cannot obtain such a set $A$ in general.
Instead,  in  the proof of 
Theorem~\ref{genthm} we will  apply
 the \emph{lattice-based absorbing method} developed recently by the first author~\cite{Han14_poly}:
What one can \emph{always} guarantee in our case is a small family of absorbing sets $\mathcal F_{abs}$ with the property that for every $m$-vertex set $S\subseteq V(H)$ such that
$\bfi_{\cP}(S)\in I_{\cP,F}^{\mu} (H)$, there are many sets in $\mathcal F_{abs}$ that do absorb $S$. 
This is made precise in Lemma~\ref{lem:abs} in Section~\ref{secabs}. We remark that to obtain $\mathcal F_{abs}$ it was crucial that condition (ii) in Theorem~\ref{genthm} holds.

Now suppose $M$ is an almost perfect $F$-packing  in $H\setminus V(\mathcal F_{abs})$. Let $U$ denote the vertices in $H\setminus V(\mathcal F_{abs})$ uncovered by $M$.
If there is a partition $S_1, \dots, S_s$ of $U$ such that $\bfi_{\cP}(S_i)\in I_{\cP,F}^{\mu} (H)$ for each $i$, then by definition of $\mathcal F_{abs}$ we can absorb the vertices in $U$ to obtain
a perfect $F$-packing in $H$. 
To find such a partition of $U$ we certainly would need that $\bfi_{\cP}(U)\in L_{\cP,F}^{\mu} (H)$. 
This is where the property that $(\mathcal P, L^{\mu}_{\mathcal P,F}(H))$ is $q$-soluble is vital: by definition this allows us to find an $F$-packing $M_1$ of size at most $q$ such that $\bfi_{\cP}(V(H)\setminus V(M_1))\in L_{\cP,F}^{\mu}(H)$. Roughly speaking, the idea is that by removing the vertices of $M_1$ from $H$ we 
now have a more `balanced' $k$-graph where (by following the steps outlined above) we do obtain a set of uncovered vertices $U$ that can be fully absorbed using the family $\mathcal F_{abs}$.
This step is a little involved; that is, some careful refinement of the uncovered set $U$ is still needed to ensure there is a partition $S_1, \dots, S_s$ of $U$ such that $\bfi_{\cP}(S_i)\in I_{\cP,F}^{\mu} (H)$ for each $i$.

Condition (iii) is applied in both the forward and backward implication of Theorem~\ref{genthm}. In particular, this is precisely the condition required to show that if $H$ has a perfect matching
then $(\mathcal P, L^{\mu}_{\mathcal P,F}(H))$ is $q$-soluble.

In the next section we prove the absorbing lemma and in Section~\ref{secgen} we prove Theorem~\ref{genthm}.

\section{Absorbing lemma}\label{secabs}

The following result guarantees our collection $\F_{abs}$ of absorbing sets in the proof of Theorem~\ref{genthm}.
\begin{lemma}[Absorbing Lemma] \label{lem:abs}
Suppose $F$ is an $m$-vertex $k$-graph and
\[
1/n \ll 1/c \ll \beta, \mu \ll 1/m, 1/t,
\]
and $H$ is a $k$-graph on $n$ vertices.
Suppose $\cP=\{ V_1, \dots, V_d\}$ is a partition of $V(H)$ such that for each $i\in [d]$, $V_i$ is $(F,\beta, t)$-closed. Then there is a family $\F_{abs}$ of disjoint $t m^2$-sets with size at most $c\log n$ such that for each $A\in \F_{abs}$, $H[A]$ contains a perfect $F$-packing and every $m$-vertex set $S$ with $\bfi_{\cP}(S)\in I_{\cP,F}^{\mu} (H)$ has at least $\sqrt{\log n}$ absorbing $(F,t m^2)$-sets in $\F_{abs}$. 
\end{lemma}
\proof
Our first task is to prove the following claim.
\begin{claim}\label{clm:abs}
Any $m$-set $S$ with $\bfi_{\cP}(S)\in I_{\cP,F}^{\mu}(H)$ has at least $\mu \beta^{m+1} n^{t m^2}$ absorbing $(F,t m^2)$-sets.
\end{claim}

\begin{proof}
For an $m$-set $S=\{y_1,\dots, y_m\}$ with $\bfi_{\cP}(S)\in I_{\cP,F}^{\mu}(H)$, we construct absorbing $(F,t m^2)$-sets for $S$ as follows. 
We first fix a copy $F'$ of $F$ with  vertex set $W=\{x_1, \dots, x_m\}$ in $H$ such that $\bfi_{\cP}(W)=\bfi_{\cP}(S)\in I_{\cP,F}^{\mu}(H)$ and $W\cap S=\emptyset$. 
Note that we have at least $\mu n^m - m n^{m-1} >\frac{\mu}2 n^m$ choices for such $F'$. 
Without loss of generality, we may assume that for all $i\in [m]$, $x_i, y_i$ are in the same part of $\cP$. 
Since $x_i$ is $(F,\beta, t)$-reachable to $y_i$, there are at least $\beta n^{t m-1}$ $(t m-1)$-sets $T_i$ such that both $H[T_i\cup \{x_i\}]$ and $H[T_i\cup \{y_i\}]$ have perfect $F$-packings. 
We pick disjoint reachable $(t m-1)$-sets for each $x_i, y_i$, $i\in [m]$ greedily, while avoiding the existing vertices. 
Since the number of existing vertices is at most $t m^2+m$, we have at least $\frac{\beta}2 n^{t m-1}$ choices for such $(t m-1)$-sets in each step.
Note that  $W\cup T_1\cup \cdots \cup T_{m}$ is an absorbing set for $S$. 
First, it contains a perfect $F$-packing because each $T_i\cup \{x_i\}$ for $i\in [m]$ spans $t$ vertex-disjoint copies of $F$. 
Second, $H[W\cup T_1\cup \cdots \cup T_{m}\cup S]$ also contains a perfect $F$-packing because $F'$ is a copy of $F$ and each $T_i\cup \{y_i\}$ for $i\in [m]$ spans $t$ vertex-disjoint copies of $F$. 
There were at least $\frac{\mu}2 n^m$ choices for $W$ and at least $\frac{\beta}2 n^{t m-1}$ choices for each $T_i$.
Thus we  find at least 
$$\frac{\mu}2 n^m \times \frac{\beta^{m}}{2^m} n^{t m^2-m} \times \frac{1}{(tm^2)!} \geq \mu \beta^{m+1} n^{t m^2}$$ absorbing $(F,t m^2)$-sets for $S$.
\end{proof}

We pick a family $\mathcal{F}$ of $t m^2$-sets by including every $t m^2$-subset of $V(H)$ with probability $p=c n^{-t m^2} \log n$ independently, uniformly at random. Then the expected number of elements in $\mathcal{F}$ is $p\binom{n}{t m^2}\le \frac{c}{t m^2}\log n$ and the expected number of intersecting pairs of $t m^2$-sets is at most
\[
p^2\binom{n}{t m^2} \cdot t m^2 \cdot \binom{n}{t m^2-1} \le \frac{c^2 (\log n)^2}{n}=o(1).
\]
Then by Markov's inequality, with probability at least $1-1/(t m^2)-o(1)$, $\mathcal{F}$ contains at most $c \log n$ sets and they are pairwise vertex disjoint.

For every $m$-set $S$ with $\bfi_{\cP}(S)\in I_{\cP,F}^{\mu}(H)$, let $X_S$ be the number of absorbing sets for $S$ in $\mathcal{F}$. Then by Claim \ref{clm:abs},
\[
\mathbb{E}(X_S) \ge p {\mu \beta^{m+1}} n^{t m^2} = {\mu \beta^{m+1} c\log n}.
\]
By Chernoff's bound, 
\[
\mathbb{P}\left( X_S\le \frac12 \mathbb{E}(X_S) \right) \le \exp\left\{ -\frac18 \mathbb{E}(X_S) \right\} \le \exp\left\{ -\frac{\mu \beta^{m+1} c \log n}{8} \right\}= o(n^{-m}),
\]
since $1/c \ll \beta,\mu \ll 1/m$.
Thus, with probability $1-o(1)$, for each $m$-set $S$ with $\bfi_{\cP}(S)\in I_{\cP,F}^{\mu}(H)$, there are at least
\[
\frac12\mathbb{E}(X_S) \ge \frac{\mu \beta^{m+1} c\log n}{2} > \sqrt{\log n}
\]
absorbing sets for $S$ in $\F$. We obtain $\F_{abs}$ by deleting the elements of $\mathcal{F}$ that are not absorbing sets for any $m$-set $S$ and thus $|\F_{abs}|\le |\F|\le c\log n$.
\endproof

\section{Proof of Theorem~\ref{genthm}}\label{secgen}

\subsection{Proof of the forward implication of Theorem~\ref{genthm}}
If $H$ contains a perfect $F$-packing $M$, then $\bfi_{\cP}(V(H)\setminus V(M))=\mathbf{0}\in L_{\cP,F}^{\mu}(H)$.
We will show that there exists an $F$-packing $M'\subset M$ of size at most $q$ such that $\bfi_{\cP}(V(H)\setminus V(M'))\in L_{\cP,F}^{\mu}(H)$ and thus $(\cP, L_{\cP,F}^{\mu}(H))$ is $q$-soluble.
Indeed, suppose $M'\subset M$ is a minimum $F$-packing such that $\bfi_{\cP}(V(H)\setminus V(M'))\in L_{\cP,F}^{\mu}(H)$ and $|M'|=m'\ge q$.
Let $M'=\{e_1, \dots, e_{m'}\}$ and consider the $m'+1$ partial sums
\[
\sum_{i=1}^{j}\bfi_{\cP}(e_i)+L_{\cP,F}^{\mu}(H) = \sum_{i=1}^{j} R_{Q(\cP, L_{\cP,F}^{\mu}(H))}(e_i),
\]
for $j=0,1,\dots, m'$.
Since $|Q(\cP, L_{\cP,F}^{\mu}(H))|\le q\le m'$, two of the sums must be equal.
That is, there exists $0\le j_1< j_2\le m'$ such that
\[
\sum_{i=j_1+1}^{j_2}\bfi_{\cP}(e_i) \in L_{\cP,F}^{\mu}(H).
\]
So the $F$-packing $M'':=M'\setminus \{e_{j_1+1},\dots, e_{j_2}\}$ satisfies that $\bfi_{\cP}(V(H)\setminus V(M''))\in L_{\cP,F}^{\mu}(H)$ and $|M''|< |M'|$, a contradiction.
\endproof

\subsection{Proof of the backward implication of Theorem~\ref{genthm}}

Suppose $I$ is a set of $m$-vectors of $\mathbb{Z}^d$ and $J$ is a (finite) set of vectors such that any $\bfi\in J$ can be written as a linear combination of vectors in $I$, namely, there exist $a_\bfv(\bfi)\in \mathbb{Z}$ for all $\bfv\in I$, such that
\begin{equation*}
\bfi=\sum_{\bfv\in I}a_{\bfv}(\bfi) \bfv. 
\end{equation*}
We denote by $C(d, m, I, J)$ as the maximum of $|a_{\bfv}(\bfi)|, \bfv\in I$ over all $\bfi\in J$.

The proof of  the backward implication of Theorem~\ref{genthm}
consists of a few steps.
We first fix an $F$-packing $M_1$, a $q$-solution of $(\cP, L_{\cP,F}^{\mu}(H))$.
We apply Lemma \ref{lem:abs} to $H$ and get a family $\F_{abs}$ of $t m^2$-sets of size at most $c\log n$.
Let $\F_{0}$ be the subfamily of $\F_{abs}$ that do not intersect $V(M_1)$.
Next we find a set $M_2$ of disjoint copies of $F$, which includes (constantly) many copies of $F$ for each $m$-vector in $I_{\cP,F}^{\mu}(H)$.
Now by definition of $\delta (F,\ell,D)$, in $H[V\setminus (V(\F_0)\cup V(M_1\cup M_2))]$ we find an $F$-packing $M_3$ covering all but a set $U$ of at most $D$ vertices.
The remaining job is to `absorb' the vertices in $U$.
Roughly speaking, by the solubility condition, we can release some copies of $F$ in some members of $\F_0$ and $M_3$, such that the set $Y\supseteq U$ of uncovered vertices satisfies that $\bfi_\cP(Y)\in L_{\cP,F}^\mu(H)$.
Furthermore, by releasing some copies of $F$ in $M_2$, we can partition the new set of uncovered vertices as a collection of $m$-sets $S$ such that $\bfi_\cP(S)\in I_{\cP,F}^{\mu}(H)$ for each $S$.
Then we can finish the absorption by the absorbing property of $\F_0$.

\begin{proof}[Proof of the backward implication of Theorem~\ref{genthm}]
Define an additional constant $C>0$ so that 
\[
1/n_0 \ll 1/C\ll \beta, \mu .
\]
Let $H$ be as in the statement of the theorem. 
Moreover, assume that $(\cP, L_{\cP,F}^{\mu}(H))$ is $q$-soluble.
We first apply Lemma~\ref{lem:abs} to $H$ and get a family $\F_{abs}$ of disjoint $t m^2$-sets of size at most $C\log n$ such that every $m$-set $S$ of vertices with $\bfi_{\cP}(S)\in I_{\cP,F}^{\mu}(H)$ has at least $\sqrt{\log n}$ absorbing $(F,t m^2)$-sets in $\F_{abs}$. 

Since $(\cP, L_{\cP,F}^{\mu}(H))$ is $q$-soluble, there exists an $F$-packing $M_1$ of size at most $q$ such that $\bfi_{\cP}(V(H)\setminus V(M_1))\in L_{\cP,F}^{\mu}(H)$. 
Note that $V(M_1)$ may intersect $V(\F_{abs})$ in at most $qm$ absorbing sets of $\F_{abs}$. 
Let $\F_{0}$ be the subfamily of $\F_{abs}$ obtained from removing the $t m^2$-sets that intersect $V(M_1)$. 
Let $M_0$ be the perfect $F$-packing on $V(\F_0)$ that is the union of the perfect $F$-packings on each member of $\F_0$. 
Note that every $m$-set $S$ of vertices with $\bfi_{\cP}(S)\in I_{\cP,F}^{\mu}(H)$ has at least $\sqrt{\log n} - qm$ absorbing sets in $\F_{0}$.

Next we want to `store' some copies of $F$ for each $m$-vector in $I_{\cP,F}^{\mu}(H)$ for future use. 
More precisely, let $J$ be the set of all $m'$-vectors in $L_{\cP,F}^{\mu}(H)$ such that $0\le m'\le qm+D$ and set $C':=C(d, m, I_{\cP,F}^{\mu}(H), J)$.
We find an $F$-packing $M_2$ in $H\setminus V(M_{0}\cup M_1)$ which contains $C'$ copies $F'$ of $F$ with $\bfi_{\cP}(F')=\bfi$ for every $\bfi\in I_{\cP,F}^{\mu}(H)$. 
So $|M_2|\le \binom{m+d-1}{m}C'$ and the process is possible because $H$ contains at least $\mu n^m$ copies of $F$ for each $\bfi\in I_{\cP,F}^{\mu}(H)$ and $|V(M_0\cup M_{1}\cup M_2)|\le t m^2 C\log n + q m + \binom{m+d-1}{m}C' m < \mu n$.

\medskip
Let $H':=H\setminus V(M_0\cup M_{1} \cup M_2)$ and $n':=|H'|$. So $n'\ge n - \mu n$ and
\[
\delta_{\ell}(H') \ge \delta_{\ell}(H) - \mu n^{k-\ell} \ge  (\delta(F,\ell ,D)+\gamma /2) \binom{n'-\ell}{k-\ell}.
\]
By the definition of $\delta(F,\ell ,D)$ we have an $F$-packing $M_3$ in $H'$ covering all but at most $D$ vertices.
Let $U$ be the set of vertices in $H'$  uncovered by $M_3$.

Let $Q:=Q(\cP, L_{\cP,F}^{\mu}(H))$.
Recall that $\bfi_{\cP}(V(H)\setminus V(M_1))\in L_{\cP,F}^{\mu}(H)$. 
Note that by definition, the index vectors of all copies of $F$ in $M_2$ are in $I_{\cP,F}^{\mu}(H)$. So we have $\bfi_{\cP}(V(H)\setminus V(M_1\cup M_2))\in L_{\cP,F}^{\mu}(H)$, namely, $R_Q(V(H)\setminus V(M_1\cup M_2))=\mathbf{0}+ L_{\cP,F}^{\mu}(H)$. 
Thus,
\[
 \sum_{F' \in M_{0}\cup M_3} R_Q(V(F'))  +R_Q(U) = \mathbf{0}+ L_{\cP,F}^{\mu}(H).
\]
Suppose $R_Q(U)=\bfv_{0}+L_{\cP,F}^{\mu}(H)$ for some $\bfv_0\in L_{\max}^{d}$; so
\[
\sum_{F' \in M_{0}\cup M_3} R_Q(V(F'))  = -\bfv_0+ L_{\cP,F}^{\mu}(H).
\]

\begin{claim}\label{clm:35}
There exist $F_1, \dots, F_{p} \in M_{0}\cup M_3$ for some $p\le q-1$ such that 
\begin{equation}\label{eq:RQ}
\sum_{i\in [p]} R_Q(V(F_i))  = -\bfv_0+ L_{\cP,F}^{\mu}(H).
\end{equation}
\end{claim}

\begin{proof}
Assume to the contrary that $F_1, \dots, F_{p} \in M_{0}\cup M_3$ is a minimum set of copies of $F$ such that \eqref{eq:RQ} holds and $p\ge q$.
Consider the $p+1$ partial sums $\sum_{i\in [j]} R_Q(V(F_i))$ for $j=0,1,\dots, p$, where the sum equals $\mathbf{0} + L_{\cP,F}^{\mu}(H)$ when $j=0$.
Since $|Q|\le q$,  two of the partial sums must be equal, that is, there exist $0\le p_1<p_2\le p$ such that $\sum_{p_1< i\le p_2} R_Q(V(F_i)) = \mathbf{0} + L_{\cP,F}^{\mu}(H)$.
So we get a smaller collection of copies of $F$ in $M_0\cup M_3$ such that \eqref{eq:RQ} holds, a contradiction.
\end{proof}

So we have $\sum_{i\in [p]} \bfi_{\cP}(V(F_i))+\bfi_{\cP}(U)\in L_{\cP,F}^{\mu}(H)$. Let $Y:=\bigcup_{i\in [p]}V(F_i)\cup U$ and thus $|Y|\le m p+D\leq mq+D$.
We now complete the perfect $F$-packing by absorption. Since $\bfi_{\cP}(Y)\in L_{\cP,F}^{\mu}(H)$, we have the following equation
\[
\bfi_{\cP}(Y) = \sum_{\bfv\in I_{\cP,F}^{\mu}(H)} a_{\bfv} \bfv,
\]
where $a_{\bfv}\in \mathbb{Z}$ for all $\bfv\in I_{\cP,F}^{\mu}(H)$. 
Since $|Y|\le qm+D$, by the definition of $C'$, we have $|a_{\bfv}|\le C'$ for all $\bfv\in I_{\cP,F}^{\mu}(H)$.
Noticing that $a_{\bfv}$ may be negative, we can assume $a_{\bfv}=b_{\bfv} - c_{\bfv}$ such that one of $b_{\bfv}, c_{\bfv}$ is $|a_{\bfv}|$ and the other is zero for all $\bfv\in I_{\cP,F}^{\mu}(H)$. So we have
\[
\sum_{\bfv\in I_{\cP,F}^{\mu}(H)} c_{\bfv} \bfv + \bfi_{\cP}(Y) = \sum_{\bfv\in I_{\cP,F}^{\mu}(H)} b_{\bfv} \bfv.
\]
This equation means that given any family $\F$ consisting of disjoint $\sum_{\bfv}c_{\bfv}$ $m$-sets $W_1^{\bfv},\dots, W_{c_{\bfv}}^{\bfv}\subseteq V(H)\setminus Y$ for $\bfv\in I_{\cP,F}^{\mu}(H)$ such that $\bfi_{\cP}(W_{i}^{\bfv})=\bfv$ for all $i\in [c_{\bfv}]$, we can regard $V(\F)\cup Y$ as the union of $b_{\bfv}$ $m$-sets $S_1^{\bfv},\dots, S_{b_{\bfv}}^{\bfv}$ such that  $\bfi_{\cP}(S_j^{\bfv})=\bfv$, $j\in [b_{\bfv}]$ for all $\bfv\in I_{\cP,F}^{\mu}(H)$.
Since $c_{\bfv}\le C'$ for all $\bfv$ and $V(M_2)\cap Y=\emptyset$, we can choose the family $\F$ as a subset of $M_2$. 
In summary, starting with the $F$-packing $M_0\cup M_1\cup M_2\cup M_3$ leaving $U$ uncovered, we delete the copies $F_1, \dots, F_{\ell}$ of $F$ from $M_{0}\cup M_3$ given by Claim~\ref{clm:35} and then leave $Y=\bigcup_{i\in [p]}V(F_i)\cup U$ uncovered.
Then we delete the family $\F$ of copies of $F$ from $M_2$ and leave $V(\F)\cup Y$ uncovered.
Finally, we regard $V(\F)\cup Y$ as the union of at most $\binom{m+d-1}{d}C'+qm+D\le \sqrt{\log n}/2$ $m$-sets $S$ with $\bfi_{\cP}(S)\in I_{\cP,F}^{\mu}(H)$. 

Note that by definition, $Y$ may intersect at most $qm+D$ absorbing sets in $\F_{0}$, which cannot be used to absorb those sets we obtained above.
Since each $m$-set $S$ has at least $\sqrt{\log n} - qm>\sqrt{\log n}/2 + qm+D$ absorbing $(F,t m^2)$-sets in $\F_{0}$, we can greedily match each $S$ with a distinct absorbing $(F,t m^2)$-set $F_S\in \F_{0}$ for $S$.
Replacing the $F$-packing on $V(F_S)$ in $M_0$ by the perfect $F$-packing on $H[F_S\cup S]$ for each $S$ gives a perfect $F$-packing in $H$.
\end{proof}


\section{Useful tools}\label{sec6}
In this section we collect together some results that will be used in our applications of Theorem~\ref{genthm}.
When considering $\ell$-degree together with $\ell'$-degree for some $\ell'\neq \ell$, the following proposition is very useful (the proof is a standard counting argument, which we omit).

\begin{prop} \label{prop:deg}
Let $0\le \ell \le \ell ' < k$ and $H$ be a $k$-graph. If $\d_{\ell '}(H)\geq x\binom{n- \ell '}{k- \ell '}$ for some $0\le x\le 1$, then $\d_{\ell}(H)\geq x\binom{n- \ell}{k- \ell}$.
\end{prop}

For the statements of the next three results, recall the definitions introduced in Section~\ref{sub23}.
Moreover, for any $S\subseteq V(H)$, let $N(S):=\{T\subseteq V(H)\setminus S: T\cup S\in E(H)\}$, and for simplicity, we write $N(x)$ for $N(\{x\})$.

\begin{lemma}[\cite{LM1}, Lemma 4.2]\label{LM4.2}
Let $k,m \ge 2$ be integers and $\r > 0$. Let $K$ be a $k$-partite $k$-graph of order $m$. There exists $0<\a \ll \r$ such that the following holds for sufficiently large $n$. For any $k$-graph $H$ of order $n$, two vertices $x, y \in V(H)$ are $(K, \a, 1)$-reachable to each other if the number of $(k-1)$-sets $S \in N(x) \cap N(y)$ with $|N(S)| \ge \r n$ is at least $\r^2  {n \choose k-1}$.
\end{lemma}

The following lemma gives us a sufficient condition for ensuring a partition $\cP = \{V_1, \dots, V_r\}$ of a $k$-graph $H$ such that for any $i\in [r]$, $V_i$ is $(F, \beta, 2^{c-1})$-closed in $H$.

\begin{lemma}\label{lem:P}
Given $\delta'>0$, integers $c, k,m \ge 2$ and $0<\a \ll 1/c, \delta', 1/m$, there exists a constant $\beta>0$ such that the following holds for all sufficiently large $n$. 
Let $F$ be an $m$-vertex $k$-graph.
Assume $H$ is an $n$-vertex $k$-graph  and $S\subseteq V(H)$ is such that $|\tilde{N}_{F, \a, 1}(v,H)\cap S| \ge \delta' n$ for any $v\in S$. Further, suppose every set of $c+1$ vertices in $S$ contains two vertices that are $(F, \a, 1)$-reachable in $H$. 
Then in time $O(n^{2^{c-1} m+1})$ we can find a partition $\cP$ of $S$ into $V_1,\dots, V_r$ with $r\le \min\{c, 1/\delta' \}$ such that for any $i\in [r]$, $|V_i|\ge (\delta' - \a) n$ and $V_i$ is $(F, \beta, 2^{c-1})$-closed in $H$.
\end{lemma}


We will use the following simple result in the proof of Lemma~\ref{lem:P}. 
\begin{proposition}\cite[Proposition 2.1]{LM1}\label{21}
Let $F$ be a fixed $k$-graph on $m$ vertices.
For $\e, \beta>0$ and an integer $i\ge 1$, there exists a $\beta_0=\beta _0 (\e,\beta,m,i)>0$ and an integer $n_0=n_0(\e,\beta,m,i)$ satisfying the following. Suppose $H$ is a $k$-graph of order $n\ge n_0$ and there exists a vertex $x\in V(H)$ with $|\tilde{N}_{F,\beta, i}(x,H)|\ge \e n$. Then for all $0<\beta'\le \beta_0$, $\tilde{N}_{F,\beta,i}(x,H)\subseteq \tilde{N}_{F,\beta',i+1}(x,H)$.
\end{proposition}



%

Next we prove Lemma~\ref{lem:P}, whose proof is almost identical to the proof of \cite[Lemma 3.8]{Han14_poly}.

\begin{proof}[Proof of Lemma~\ref{lem:P}]
Let $\e:=\a/c$.
We choose constants satisfying the following hierarchy
\[
1/n\ll \beta = \beta_{c-1}\ll \beta_{c-2}\ll \cdots \ll \beta_1 \ll \beta_0 \ll \e \ll 1/c,\delta ',1/m.
\]

Let $F$ and $H$ be as in the statement of the lemma.
Throughout this proof, given $v\in V(H)$ and $i\in [c-1]$, we write $\tilde{N}_{F, \beta_i, 2^i}(v,H)$ as $\tilde N_{i}(v)$ for short. 
Note that for any $v\in V(H)$, $|\tilde N_0(v)|=|\tilde{N}_{F, \beta_0, 1}(v,H)|\ge |\tilde{N}_{F, \a, 1}(v,H)| \ge \delta' n$ because $\beta_0<\a$.
We also write $2^i$-reachable (or $2^i$-closed) for $(F, \beta_i, 2^i)$-reachable (or $(F, \beta_i, 2^i)$-closed). By Proposition \ref{21} and the choice of $\beta_i$s, we may assume that $\tilde {N}_{i}(v)\subseteq \tilde {N}_{i+1}(v)$ for all $0\le i<c-1$ and all $v\in V(H)$.
Hence, if $W\subseteq V(H)$ is $2^i$-closed in $H$ for some $i\le c-1$, then $W$ is $ 2^{c-1}$-closed.

We may assume that there are two vertices in $S$ that are not $2^{c-1}$-reachable to each other, as otherwise $S$ is $2^{c-1}$-closed in $H$ and we obtain the desired (trivial) partition $\cP=\{S\}$.
Let $r$ be the largest integer such that
there exist $v_1,\dots, v_{r}\in S$ such that no pair of them are $2^{c+1-r}$-reachable in $H$. 
Note that $r$ exists by our assumption and $2\le r\le c$. 
Fix such $v_1,\dots, v_{r}\in S$; by Proposition \ref{21}, we can assume that any pair of them are not $2^{c-r}$-reachable in $H$. Consider $\tilde N_{c-r}(v_i)$ for all $i\in [r]$. Then we have the following facts.
\begin{enumerate}[(i)]
\item Any $v\in S\setminus\{v_1,\dots, v_{r}\}$ must lie in $\tilde N_{c-r}(v_i)$ for some $i\in [r]$, as otherwise $v, v_1,\dots, v_{r}$ contradicts the definition of $r$. 

\item $|\tilde N_{c-r}(v_i)\cap \tilde N_{c-r}(v_j)|<\e n$ for any $i\neq j$. Indeed, otherwise there are at least
\[
\frac{\e n}{(2^{c+1-r}m-1)!} (\beta_{c-r}n^{2^{c-r}m-1} - n^{2^{c-r}m-2}) (\beta_{c-r}n^{2^{c-r}m-1} - 2^{c-r}m n^{2^{c-r}m-2})
\]
reachable $(2^{c+1-r}m-1)$-sets for $v_i, v_j$. This follows because there are at least $\e n$ vertices 
$w\in \tilde N_{c-r}(v_i)\cap \tilde N_{c-r}(v_j)$, at least $ \beta_{c-r}n^{2^{c-r}m-1} - n^{2^{c-r}m-2}\ $ reachable $(2^{c-r}m-1)$-sets $T$ for $v_i$ and $w$ that do not contain $v_j$, and at least $\beta_{c-r}n^{2^{c-r}m-1} - 2^{c-r}m n^{2^{c-r}m-2}\ $ reachable $(2^{c-r}m-1)$-sets for $v_j$ and $w$ that avoid $\{v_i\}\cup T$; finally, we divide by $(2^{c+1-r}m-1)!$ to eliminate the effect of over-counting. Since $\beta_{c+1-r}\ll \e, \beta_{c-r}, 1/c,1/m$, this gives at least $\beta_{c+1-r} n^{2^{c+1-r}m-1}$ reachable $(2^{c+1-r}m-1)$-sets for $v_i, v_j$, contradicting the assumption that $v_i, v_j$ are not $2^{c+1-r}$-reachable to each other.
\end{enumerate}
Note that (ii) and $|\tilde N_{c-r}(v_i)\cap S|\ge |\tilde{N}_{0}(v_i)\cap S| \ge \delta' n$ for $i\in [r]$ imply that $r\delta' n - \binom r2 \e n \le |S|\le n$. So we have $r\le (1+c^2 \e)/\delta'$. Since $\e \le \a \ll \delta',1/c$, we have $r\le 1/\delta'$ and thus, $r\le \min\{c, 1/\delta'\}$.

For $i\in [r]$, let $U_i:=((\tilde N_{c-r}(v_i)\cup \{v_i\})\cap S)\setminus \bigcup_{j\in [r]\setminus \{i\}} \tilde N_{c-r}(v_j)$. Note that for $i\in [r]$, $U_i$ is $2^{c-r}$-closed in $H$. Indeed, if there exist $u_1, u_2\in U_i$ that are not $2^{c-r}$-reachable to each other, then $\{u_1, u_2\}\cup (\{v_1,\dots, v_{r}\}\setminus\{v_{i}\})$ contradicts the definition of $r$.

Let $U_0 := S\setminus (U_1\cup\cdots \cup U_{r})$. By (i) and (ii), we have $|U_0|\le \binom{r}{2}\e n$. We will move each vertex of $U_0$ greedily to $U_i$ for some $i\in [r]$. For any $v\in U_0$, since $|(\tilde N_0(v)\cap S)\setminus U_0|\ge \delta' n - |U_0|\ge r\e n$, there exists $i\in [r]$ such that $v$ is 1-reachable to at least $\e n$ vertices in $U_i$. In this case we add $v$ to $U_i$ (we add $v$ to an arbitrary $U_i$ if there are more than one such $i$). 
Let the resulting partition of $S$ be $V_1,\dots, V_{r}$. Note that we have $|V_i|\ge |U_i|\ge |\tilde N_{c-r}(v_i)\cap S| - r \e n\ge |\tilde N_0 (v_i)\cap S| - c\e n \ge (\delta' - \a )n$. Observe that in each $V_i$, the `farthest' possible pairs are those two vertices both from $U_0$, which are $2^{c-r+1}$-reachable to each other. Thus, each $V_i$ is $2^{c-r+1}$-closed, so $2^{c-1}$-closed because $r\ge 2$. 

We estimate the running time as follows.
First, for every two vertices $u, v\in S$, we determine if they are $2^{i}$-reachable for $0\le i\le c-1$.
This can be done by testing if any $(2^i m -1)$-set $T\in \binom{V(H)\setminus \{u,v\}}{2^i m -1}$ is a reachable set for $u$ and $v$, namely, if both $H[T\cup \{u\}]$ and $H[T\cup \{v\}]$ have perfect $F$-packings or not, which can be checked by listing the edges on them, in constant time. 
If there are at least $\beta_{i} n^{2^{i}m-1}$ reachable $(2^{i}m-1)$-sets for $u$ and $v$, then they are $2^{i}$-reachable. 
Since we need time $O(n^{2^{c-1}m-1})$ to list all $(2^{c-1}m-1)$-sets for each pair $u, v$ of vertices, this can be done in time $O(n^{2^{c-1}m+1})$. 
Second, we search the set of vertices $v_1,\dots, v_r$ such that no pair of them are $2^{c+1-r}$-reachable for all $2\le r\le c$. With the reachability information at hand, this can be done in time $O(n^{c})$.
We then fix the largest $r$ as in the proof. If such $r$ does not exist, then we get $\cP=\{S\}$ and output $\cP$. Otherwise, we fix any $r$-set $v_1,\dots, v_r$ such that no pair of them are $2^{c+1-r}$-reachable. We find the partition $\{U_0, U_1, \dots, U_r\}$ by identifying $\tilde{N}_{c-r}(v_i)$ for $i\in [r]$, in time $O(n)$.
Finally we move vertices of $U_0$ to $U_1,\dots, U_r$, depending on $|\tilde{N}_0(v)\cap U_i|$ for $v\in U_0$ and $i\in [r]$, which can be done in time $O(n^2)$. Thus, the running time for finding a desired partition is $O(n^{2^{c-1}m+1})$.
\end{proof}

\section{Tools for Theorem~\ref{mainthm}}\label{sec7}
In the following section we prove Theorem~\ref{mainthm}. Here we collect together some useful notation and results for this proof.

Let $H$ be a $k$-graph.
In the case of perfect matchings (i.e. when $F$ is an edge) we write $(\beta, i)$-reachable, $( \beta, i)$-closed  and $\tilde{N}_{\b, i}(v,H)$ for $(F, \beta, i)$-reachable  $(F, \beta, i)$-closed and $\tilde{N}_{F,\b, i}(v,H)$ respectively.

%

The following result is a weaker version of Lemma 5.6 in~\cite{KOTo}.

\begin{lemma}\cite{KOTo} \label{lem:frac}
Let $k \geq 2$ and $1 \leq \ell \leq k-1$ be integers, and let $\eps >0$. Suppose that for some $b,c \in (0,1)$ and some $n_0 \in \mathbb N$, every $k$-graph $H$ on $n \geq n_0$ vertices with $\delta _{\ell} (H) \geq cn^{k-\ell}$ has a fractional matching of size $(b+\eps)n$. Then there exists an $n_0 ' \in \mathbb N$ such that any $k$-graph $H$ on $n \geq n'_0$ vertices with $\delta _{\ell} (H) \geq (c+ \eps) n^{k-\ell}$ contains a matching of size at least $bn$.
\end{lemma}


Note that $\delta(k, k-1, k)= c_{k,k-1}^*=1/k$ by the results in \cite{RRS09}.
The following theorem follows from \cite[Theorem 1.7]{Han15_mat} and \cite[Proposition 1.11]{Han15_mat} when $2\le \ell\le k-1$ and follows from the Strong Absorbing Lemma in~\cite[Lemma 2.4]{HPS} and the definition of $c_{k,\ell}^*$ when $\ell=1$.

\begin{theorem}\label{thm:alm_mat}
For $1\le \ell \le k-1$, $ \delta(k, \ell, k)\le \max\{1/3, c_{k,\ell}^*\}$.
\end{theorem}

In fact, it is not hard to show that $\delta(k, \ell, k)= c_{k,\ell}^*$ for any $1\le \ell \le k-1$, but Theorem~\ref{thm:alm_mat} is enough for this paper.




\section{Proof of Theorem~\ref{mainthm}}\label{sec8}

Let $\delta \in (\delta ^*,1]$ and define
\[
0 < 1/n_0 \ll 1/c \ll  \mu \ll \beta \ll \alpha ' \ll \eta \ll \a \ll \gamma \ll (\delta-\delta ^*), 1/k.
\]
Let $H$ be as in the statement of Theorem~\ref{mainthm}. Note that we may assume  $n\geq n_0$ and $k\mid n$ since else the result is trivial (recall the use of big-$O$ notation in the statement of the theorem).
So
\begin{align}\label{eq1}
\delta _\ell (H) \geq (\delta ^*+\gamma ) \binom{n-\ell}{k-\ell} \geq (1/3+\gamma ) \binom{n-\ell}{k-\ell}
\end{align}
and in particular, by  Proposition~\ref{prop:deg},
\begin{align}\label{eq2}
\delta _1 (H) \geq (1/3+\gamma ) \binom{n-1}{k-1}.
\end{align}
Notice that by (\ref{eq2}),
\begin{itemize}
\item[($*$)] Every set of three vertices of $V(H)$ contains two vertices that are $(\a,1)$-reachable. 
\end{itemize}

Note that  when $\ell =1$, we are  just in a subcase of~\cite[Theorem 1.1]{AFHRRS};
 in this case we must have a perfect matching.
So we may assume that $\ell >1$.

We now split the argument into two cases.
\subsection{There exists $v\in V(H)$ such that $|\tilde{N}_{\a, 1}(v,H)| \le \eta  n$.}\label{sec41}
In this case, we will show that $H$ must contain a perfect matching. 
Let $W:=\{v\}\cup \tilde{N}_{\a, 1}(v,H)$ and thus $|W|\le \eta n+1$.
For any two vertices $u, u'\in V(H)\setminus W$, since $u, u'\notin \tilde{N}_{\a, 1}(v,H)$, by ($*$), $u$ and $u'$ are $(\a, 1)$-reachable, i.e., $V(H)\setminus W$ is $(\a, 1)$-closed in $H$.
Let $H_1:=H\setminus W$ and $n_1:=|H_1|$. Since $\eta \ll \alpha$ we have that $V(H)\setminus W=V(H_1)$ is  $(\a/2, 1)$-closed in $H_1$.

By Lemma~\ref{lem:abs} (with $d=1$) there is a set $T \subseteq V(H_1)$ (take $T:=V(\F_{abs})$) such that $|T|\leq c k^2\log n_1$ and both $H_1[T]$ and $H_1[T\cup S]$ contain perfect matchings for \emph{any}
set $S\subseteq V(H_1)$ where $|S|\in k\mathbb N$ and $|S|\leq \sqrt{\log n_1}$. 
We greedily construct a matching $M$ in $H$ such that $|M|\leq \eta n +1$; $W \subseteq V(M)$; and $V(M)\cap T=\emptyset$.
Let $H_2:=H\setminus (V(M)\cup T)$ and $n_2:=|H_2|$. Note that $H_2$ is a subgraph of $H_1$.
By (\ref{eq1}), the definition of $\delta ^ *$ and Theorem~\ref{thm:alm_mat},
$$\delta _\ell (H_2) \geq (\delta^*  +\gamma/2 ) \binom{n_2-\ell}{k-\ell} \geq (\delta(k,\ell,k) +\gamma/2 ) \binom{n_2-\ell}{k-\ell}.$$
Thus, by definition of $\delta(k,\ell,k)$, $H_2$ contains a matching $M_1$ covering all but at most $k$ vertices of $H_2$.
Let $S$ denote the leftover set of vertices. (So $S=\emptyset$ or $|S|=k$.) Then $H[T\cup S]$ contains a perfect matching $M_2$.
Altogether, $M\cup M_1 \cup M_2$ is a perfect matching in $H$, as desired.

\subsection{Every vertex $v \in V(H)$ satisfies $|\tilde{N}_{\a, 1}(v,H)| \ge \eta n$.}\label{sec42} 
Thus, since $\alpha' \ll \alpha$, every vertex $v \in V(H)$ satisfies $|\tilde{N}_{\a ', 1}(v,H)| \ge \eta n$.
Apply Lemma~\ref{lem:P} to $H$ (with $\alpha '$, $2$, $\eta$ playing the roles of $\alpha$, $c$ and $\delta '$ respectively) to
 find a partition $\cP$ of $V(H)$ into $V_1,\dots, V_r$ with $r\le 2$ such that for any $i\in [r]$, $|V_i|\ge \eta n/2$ and $V_i$ is $(\beta, 2)$-closed in $H$, in time $O(n^{2 k+1})$.


Our aim is to apply Theorem~\ref{genthm} to $H$.
First, by Theorem~\ref{thm:alm_mat} and \eqref{eq1}, we have that $\delta_{\ell}(H) \ge (\delta(k,\ell, k)+\r)\binom{n-\ell}{k-\ell}$.
Second, by definition, $\cP$ is an $(E, \beta, 2, \eta/2)$-good partition of $V(H)$, where $E$ is a $k$-graph  on $k$ vertices consisting of a single edge.

Write $L:=L^{\mu}_{\mathcal P,E}(H)$ and $Q:=Q(\mathcal P, L^{\mu}_{\mathcal P,E}(H))$.
We will show that $|Q|\leq k$.
Clearly, if $r=1$, then $|Q|=1$.
So we may assume $r=2$.
First assume that $I^{\mu}_{\mathcal P,E}(H)$ contains two distinct elements, say, $(a, b), (a', b')\in I^{\mu}_{\mathcal P,E}(H)$ with $a\neq a'$.
Thus $(a-a', b-b')=(a-a',a'-a)\in L^{\mu}_{\mathcal P,E}(H)$.
Any coset $(x,y)+L$ in $Q$ must contain some element $(x',y')$ so that $x'+y'=k$. Consider two 
vectors $(n_1, n_2), (n_1', n_2') \in L^2 _{max}$ where $n_1+n_2=n'_1+n'_2=k$. 
If $n_1\equiv n_1' \pmod{|a-a'|}$ then these two vectors lie in the same coset in $Q$. (Indeed, by adding a multiple of $(a-a',a'-a)$ to $(n_1, n_2)$ one can obtain $(n_1', n_2')$.)
Altogether this implies 
 there are at most $|a-a'|$ cosets, i.e., $|Q|\leq |a-a'|\le k$.

Second, assume that $I^{\mu}_{\mathcal P,E}(H)$ contains exactly one element, say $I^{\mu}_{\mathcal P,E}(H) = \{(a,b)\}$, where $a+b=k$.
Note that it must hold that $a\ge \ell$ and $b\ge \ell$.
Indeed, if $a<\ell$, then the number of edges that contain an $\ell$-set of index vector $(\ell, 0)$ is at most $2^k\mu n^k$.
Thus, by averaging and since $\mu \ll \eta \ll 1/k$, there exists an $\ell$-set $S$ of index vector $(\ell, 0)$ such that $d_H(S)\le \binom{k}{\ell}2^k \mu n^k/\binom{|V_1|}{\ell}\le \sqrt\mu n^{k-\ell} < \delta_{\ell}(H)$, a contradiction.
The same argument shows that $b\ge \ell$.
Then for $0\le \ell_1\le \ell$, consider the $\ell$-vectors $(\ell_1, \ell_2)$.
By averaging, for each $0\le \ell_1\le \ell$, there exists an $\ell$-set $S_{\ell_1}$ of index vector $(\ell_1, \ell_2)$ such that
\[
d_H(S_{\ell_1})\le \binom{|V_1|-\ell_1}{a-\ell_1}\binom{|V_2| - \ell_2}{b-\ell_2}+\frac{\binom{k}{\ell} 2^k \mu n^k}{\binom{|V_1|}{\ell_1}\binom{|V_2|}{\ell_2}} \le \binom{|V_1|}{a-\ell_1}\binom{|V_2|}{b-\ell_2} + \sqrt\mu n^{k-\ell}.
\]
Recall the identity $\sum_{0\le i\le t}\binom{n_1}{i}\binom{n_2}{t-i} = \binom{n_1+n_2}{t}$, so we have
\[
\sum_{0\le \ell_1\le \ell} d_H(S_{\ell_1}) \le \binom{n}{k-\ell} + k \sqrt\mu n^{k-\ell} \le \binom{n-\ell}{k-\ell} + 2k \sqrt\mu n^{k-\ell}.
\]
Since $\ell \ge 2$ and $a, b\ge \ell$, the above sum contains at least three terms.
As $\mu \ll \r \ll 1/k$, there exists $\ell_1$ such that $d_H(S_{\ell_1})\le \frac13\binom{n-\ell}{k-\ell} + 2k \sqrt\mu n^{k-\ell} < (\frac13+\r)\binom{n-\ell}{k-\ell}$, contradicting~\eqref{eq1}.
That is, the case when $I^{\mu}_{\mathcal P,E}(H)$ contains one element does not occur.

Therefore we can apply Theorem~\ref{genthm} to $H$ with $D=q=k$, $t=2$ and $c=\eta/2$ and thus conclude that $H$ contains a perfect matching if and only if $(\mathcal P, L^{\mu}_{\mathcal P,E}(H))$ is $k$-soluble.

{\bf The algorithm.}
Now we state our algorithm. First, for every two vertices $u, v\in V(H)$, we determine if they are $(\a, 1)$-reachable, which can be done by testing if any $(k-1)$-set is a reachable set in time $O(n^{k-1})$.
So this step can be done in time $O(n^{k+1})$.
Then we check if $|\tilde{N}_{\a, 1}(v,H)| \ge \eta n$ for every $v\in V(H)$.
With the reachability information, this can be tested in time $O(n^2)$.
If $|\tilde{N}_{\a, 1}(v,H)| < \eta n$ for some $v\in V(H)$, then we output PM and halt.
Otherwise we run the algorithm with running time $O(n^{2k+1})$ provided by Lemma~\ref{lem:P} and get a partition $\cP$.
By Theorem~\ref{genthm}, it remains to test if $(\mathcal P, L^{\mu}_{\mathcal P,E}(H))$ is $k$-soluble.
This can be done by testing whether any matching $M$ of size at most $k$ is a solution of $(\mathcal P, L^{\mu}_{\mathcal P,E}(H))$, in time $O(n^{k^2})$.
If there is a solution $M$ for $(\mathcal P, L^{\mu}_{\mathcal P,E}(H))$, output PM; otherwise output NO.
The overall running time is $O(n^{k^2})$.


\section{The perfect graph packing result}\label{sec9}
In this section we prove Theorem~\ref{thm:Ftil}.
Let $F$ be an $m$-vertex $k$-chromatic graph.
By the definition of $\chi_{cr}(F)$, we have
\begin{equation}\label{eqx}
\frac1{\chi_{cr}(F)} = \frac{m-\sigma(F)}{(k-1)m} \le \frac{m-1}{(k-1)m}.
\end{equation}

We will apply the following variant of Lemma~\ref{LM4.2}, which can be easily derived from the original version by defining a $k$-graph $G'$ where each $k$-set forms a hyperedge if and only if it spans a copy of $K_k$ in $G$.
For any vertex $u\in V(G)$, let $W(u)$ denote the collection of $(k-1)$-sets $S$ such that $S\subseteq N(u)$ and such that $S$ spans a clique in $G$.
For a set $T\subseteq V(G)$, let $N(T):=\bigcap_{v\in T}N(v)$.

\begin{lemma}\cite{LM1}\label{agood}
Let $k,m \in \mathbb N$ and $\r'>0$. There exists $\a=\a (k,m, \r')>0$
 such that the following holds for sufficiently large $n$. 
Let $F$ be a $k$-chromatic graph on $m$ vertices. For any $n$-vertex graph $G$, two vertices $x, y \in V(G)$ are $(F,\a,1)$-reachable if the number of $(k-1)$-sets $S\in W(x)\cap W(y)$ with $|N(S)|\ge \r' n$ is at least $(\r')^2 \binom{n}{k-1}$.
\end{lemma}
We apply Lemma~\ref{agood} to prove the following result.
\begin{proposition}\label{prop:Nv}
Let $k,m,n \geq 2$ be integers and $\a, \r>0$ where
 $0< 1/n \ll \a \ll \r \ll 1/m,1/k$.
Let $F$ be a $k$-chromatic graph on $m$ vertices and
 let $G$ be an $n$-vertex graph with $\delta(G)\ge (1-1/\chi_{cr}(F)+\r) n$. Then for any $v\in V(G)$, $|\tilde{N}_{F,\a, 1}(v, G)| \ge (1/m+\r/2)n$.
\end{proposition}

\begin{proof}
For each $(k-1)$-set $S$, since $\delta(G)\ge (1-1/\chi_{cr}(F)+\r) n$, by~\eqref{eqx} we have $|N(S)|\ge (1/m + (k-1)\r)n$.
Then by Lemma~\ref{agood}, for any distinct $u, v\in V(G)$, $u\in \tilde{N}_{F,\a, 1}(v, G)$ if $|W(u)\cap W(v)|\ge \r^2 \binom{n}{k-1}$.
By double counting, we have
\[
 \sum_{S\in W(v)} \left (|N(S)|-1\right) \leq |\tilde{N}_{F,\a, 1}(v, G)|\cdot |W(v)|+n\cdot \r^2 \binom{n}{k-1}.
\]
Note that any $S$ in the above inequality is a $(k-1)$-set, thus $|N(S)| \ge (1/m + (k-1)\r)n$. 
On the other hand, using the minimum degree condition, it is easy to see that $|W(v)|\ge \frac{1}{m^{k-1}}\binom{n}{k-1}$. Since $\r \ll 1/m, 1/k$, we have
\[
|\tilde{N}_{F,\a, 1}(v, G)|\geq (1/m + (k-1)\r)n-1 - \frac{\r^2 n^k}{|W(v)|}\ge (1/m + \r/2)n.\qedhere
\]
\end{proof}

The following proposition shows that $|Q(\cP, L_{\cP,F}^{\mu}(G))|$ is bounded from above.

\begin{proposition}\label{prop:group}
Let $t,r,k,m ,n_0 \in \mathbb N$ where $k \geq 2$ and let $\beta, \mu ,\r>0$
so that
\[
1/n_0 \ll \beta, \mu \ll \r \ll 1/m, 1/t.
\]
Let $F$ be an unbalanced $m$-vertex $k$-chromatic graph.
Suppose $G$ is a graph on $n\ge n_0$ vertices such that $\delta(G)\ge (1-1/\chi_{cr}(F)+\r) n$ with an $(F,\beta, t, 1/m)$-good partition $\cP$ where $|\cP|=r$. 
Then $|Q(\cP, L_{\cP,F}^{\mu}(G))|\le (2 m-1)^r$.
\end{proposition}

We need the following simple counting result, which, for example, follows from the result of Erd\H{o}s \cite{erdos} on supersaturation. 

\begin{proposition}\label{prop:erdos}
Given $\r'> 0$, $\ell_1, \dots, \ell_ k\in \mathbb{N}$, there exists $\mu>0$ such that the following holds for sufficiently large $n$. Let $T$ be an $n$-vertex graph with a vertex partition $V_1 \cup \dots \cup V_d$. Suppose $i_1, \dots, i_k\in [d]$ are not necessarily distinct and $T$ contains at least $\r' {n}^{k}$ copies of $K_k$ with vertex set $\{ v_1, \dots, v_k \}$ such that $v_1\in V_{i_1}$, $\dots, v_k\in V_{i_k}$. Then $T$ contains at least $\mu {n}^{\ell_1+\cdots + \ell_k}$ copies of $K^{(2)}(\ell_1, \dots, \ell_ k)$ whose $j$th part is contained in $V_{i_j}$ for all $j\in [k]$.
\end{proposition}

We write $\bfu_j$ for the `unit' 1-vector that has 1 in coordinate $j$ and 0 in all other coordinates.

\begin{proof}[Proof of Proposition~\ref{prop:group}]
Write $L:=L_{\cP,F}^{\mu}(G)$.
It suffices to show that for any element $\bfv\in L_{\max}^{r}$, there exists $\bfv'=(v_1',\dots, v_r')\in L_{\max}^{r}$
 such that $-(m-1)\le v_i'\le m-1$ for all $i\in [r]$ and $\bfv +L= \bfv'+ L$. In particular, the number of such $\bfv'$ is at most $(2m-1)^r$.
Since $F$ is unbalanced, there exists a $k$-colouring with colour class sizes $a_1\le \cdots \le a_k$ and $a_1< a_k$. 
Set $a:=a_{k} - a_{1}<m$.

Let $\cP = \{V_1, \dots, V_r\}$ be the partition of $V(G)$ given in the statement of the proposition.
Define a graph $P$ on the vertex set $[r]$ such that $(i, j)\in E(P)$ if and only if $e(G[V_i, V_j])\ge \r n^2$.
We claim that if $i$ and $j$ are connected by a path in $P$, then $a(\bfu_i-\bfu_j)\in L$.
Indeed, first assume that $(i,j)\in E( P)$.
For each edge $u v$ in $G[V_i, V_j]$, since
\[
\delta(G)\ge (1-1/\chi_{cr}(F)+\r) n\stackrel{(\ref{eqx})}{\ge} \left(1- \frac{m-1}{(k-1)m} +\r\right)n,
\]
it is easy to see that $u v$ is contained in at least $\frac{1}{m^{k-2}} \binom{n}{k-2}$ copies of $K_k$ in $G$.
So there are at least $\r n^2 \cdot \frac{1}{m^{k-2}} \binom{n}{k-2}/\binom{k}{2}$ copies of $K_k$ in $G$ intersecting both $V_i$ and $V_j$.
By averaging, there exists a $k$-array $(i_1, \dots, i_k)$, $i_j\in [r]$ where $i_1=i$ and $i_k=j$ such that $G$ contains at least
\[
\frac{1}{r^{k-2}}\r n^2 \cdot \frac{1}{m^{k-2}} \binom{n}{k-2}/\binom{k}{2} \ge \frac{\r}{m^{k-2}r^{k-2}k!}n^{k}
\]
copies of $K_k$ with vertex set $\{ v_1, \dots, v_k \}$ such that $v_1\in V_{i_1}$, $\dots, v_k\in V_{i_k}$.
By applying Proposition~\ref{prop:erdos} with $\ell_i:=a_i$ for each $i\in [k]$, we get that there are at least $\mu n^{m}$ copies of $K^{(2)}(a_1, \dots, a_ k)$ in $G$ whose $j$th part is contained in $V_{i_j}$ for all $j\in [k]$.
We apply Proposition~\ref{prop:erdos} again, this time with $\ell_i:=a_i$ for all $2\le i\le k-1$ and $\ell_{1}:=a_{k}$, $\ell_{k}:=a_{1}$ and thus conclude that there are at least $\mu n^{m}$ copies of $K^{(2)}(a_k, a_2,\dots, a_{k-1}, a_1)$ (with $a_1$ and $a_k$ exchanged) in $G$ whose $j$th part is contained in $V_{i_j}$ for all $j\in [k]$.
Taking subtraction of index vectors of these two types of copies gives that $a(\bfu_i-\bfu_j)\in L$.
Furthermore, note that if $i$ and $j$ are connected by a path in $P$, we can apply the argument above to every edge in the path and conclude that $a(\bfu_i-\bfu_j)\in L$, so the claim is proved.

We now distinguish two cases.

\medskip
\noindent \textbf{Case 1: $k\ge 3$.}
In this case, we first show that $P$ is connected.
Indeed, we prove that for any bipartition $A\cup B$ of $[r]$, there exists $i\in A$ and $j\in B$ such that $(i,j)\in E(P)$.
Let $V_A:=\bigcup_{i\in A} V_i$ and $V_B:=\bigcup_{j\in B}V_j$.
Without loss of generality, assume that $|V_A|\le n/2$.
Since $\delta(G)\ge \frac{1+(k-2)m}{(k-1)m} n \ge (1/2 + 1/(2m))n$, the number of edges in $G$ that are incident to $V_A$ is at least
\[
|V_A| \cdot \left(\frac12 + \frac{1}{2m} \right)n - \binom{|V_A|}{2} \ge \binom{|V_A|}{2} + \frac{n}{4m}|V_A| \ge \binom{|V_A|}{2} + \r n^2 |A| |B|,
\]
where the last inequality follows since $|A| |B|\le r^2/4$,  $|V_i|\ge n/m$ for each $i\in [r]$ and $\gamma \ll 1/m$.
By averaging, there exists $i\in A$ and $j\in B$ such that $e(G[V_i, V_j])\ge \r n^2$ and thus $(i,j)\in E(P)$.

Now let $\bfv=(v_1,\dots, v_r)\in L_{\max}^{r}$.
We fix an arbitrary $m$-vector $\bfw\in L$ and let $\bfv_1 := \bfv - ( |\bfv|/m ) \bfw$. So  $|\bfv_1| =0$ and $\bfv_1+L=\bfv +L$.
Since $P$ is connected, the claim above implies that for any $i,j\in [r]$, $a(\bfu_i-\bfu_j)\in L$.

We now apply the following algorithm to $\bfv_1$.
Suppose $v^1 _i$ is the coordinate of $\bfv _1$ with $|v^1 _i|$ maximised.
If $|v^1 _i|\leq a \leq m-1$ we terminate the algorithm. Otherwise, since 
$|\bfv_1| =0$, there is some coordinate $v^1 _j$ of $\bfv _1$ 
where the difference between $v^1_i$ and $v^1 _j$ is more than $a$. 
We now redefine $\bfv _1$ by (i) subtracting $a(\bfu_i-\bfu_j)\in L$ from $\bfv _1$ if $v^1 _i> a$ or (ii) adding $a(\bfu_i-\bfu_j)\in L$ to $\bfv_1$ if $v^1 _i< a$.
Note that still $|\bfv_1| =0$.

We repeat this algorithm until we obtain a vector $\bfv '=(v'_1, \dots, v'_r)$ so that 
$|\bfv '| =0$ and $-(m-1)\le -a\le v_i'\le a \leq m-1$ for all $i\in [r]$. Note
$\bfv'$ was obtained from $\bf v_1$ by repeatedly adding and subtracting elements
of $L$ to $\bfv_1$. Since initially $\bfv_1+L=\bfv +L$ we have that 
$\bfv '+L=\bfv +L$, as desired.

\medskip
\noindent \textbf{Case 2: $k= 2$.}
In this case we cannot guarantee that $P$ is connected (we may even have some isolated vertices).
First let $i$ be an isolated vertex in $P$.
By the definition of $P$, we know that $e(G[V_i, V\setminus V_i]) \le (r-1)\r n^2$.
Since $\delta(G)\ge n/m$, 
\[
e(G[V_i]) \ge \frac12 (|V_i| n/m - (r-1)\r n^2) \ge \frac1{4m} |V_i|^2.
\]
Applying Proposition~\ref{prop:erdos} on $V_i$ shows that there are at least $\mu n^m$ copies of $K^{(2)}(a_1, a_2)$ in $G[V_i]$, i.e., $m\bfu_i\in L$.
Second, if $(i,j)\in E(P)$, then applying Proposition~\ref{prop:erdos} to $G[V_i, V_j]$ gives that $a_1\bfu_i + a_2\bfu_j\in L$.
So in both cases, for any component $C$ in $P$, there exists an $m$-vector $\bfw\in L$ such that $\bfw |_{[d]\setminus C}=\textbf{0}$.

Now let $\bfv=(v_1,\dots, v_r)\in L_{\max}^{r}$.
Consider the connected components $C_1, C_2, \dots, C_q$ of $P$, for some $1\le q\le r$.
By the conclusion in the last paragraph, there exists $\bfv_1\in L_{\max}^{r}$ such that $\bfv - \bfv_1\in L$ (i.e. $\bfv +L=\bfv_1 +L$) and for each component $C_i$, $0\le |\bfv_1|_{C_i}|\le m-1$. (We obtain $\bfv_1$ from $\bfv$ by adding or subtracting from it a multiple of the vector $\bfw$ given by the last paragraph, for each component $C$.)
By using an analogous algorithm to the one in Case 1, we can obtain the desired vector $\bfv'$ 
from $\bfv _1$. Indeed, using the vectors $\bfw$ given by the last paragraph,  within each nontrivial component $C_i$, we can `balance' the coordinates, as in Case 1. In particular, note that if $(i,j)\in E(P)$ then both $a_1\bfu_i + a_2\bfu_j, a_2\bfu_i + a_1\bfu_j\in L$ and so $a(\bfu_i - \bfu_j)\in L$.
\end{proof}

Now we are ready to prove Theorem~\ref{thm:Ftil}.

\begin{proof}[Proof of Theorem~\ref{thm:Ftil}]
We first note that it suffices to prove Theorem~\ref{thm:Ftil} in the case when $F$
is unbalanced. Indeed, if $F$ is balanced then $\chi (F)=\chi _{cr} (F)$ and so
the result follows (trivially) from Theorem~\ref{thm:AY}.

Given any $\delta \in (1-1/\chi _{cr}(F),1]$ let $\mu,\alpha, \gamma >0$ so that $0<\mu \ll\alpha\ll \gamma \ll (\delta -1+1/\chi _{cr}(F)), 1/m,1/k$.
Apply Lemma \ref{lem:P} with $c:=m^{k-1}$, $\delta':=1/m+\r /2$ to obtain some $\beta>0$. We may assume $\beta \ll \alpha$. Finally choose $n_0 \in \mathbb N$ such that $1/n_0 \ll \beta, \mu$.
Altogether we have
\[
1/n_0\ll \beta, \mu \ll \a \ll \r \ll (\delta -1+1/\chi _{cr}(F)), 1/m,1/k.
\]
Let $G$ be an $n$-vertex graph as in the statement of Theorem~\ref{thm:Ftil}. We may assume that $n\geq n_0$ and $m$ divides $n$ since else the result is trivial.
Note that $\delta(G)\ge \delta n \geq (1-1/\chi_{cr}(F)+\r) n$.

By Proposition \ref{prop:Nv}, for any $v\in V(G)$, $|\tilde{N}_{F,\a, 1}(v, G)| \ge \delta' n$.
The degree condition and Lemma~\ref{agood} imply that, for distinct $u, v\in V(G)$, $u$ and $v$ are $(F,\a, 1)$-reachable if $|W(u)\cap W(v)|\ge \r^2 \binom{n}{k-1}$. 
Further, for any $u \in V(G)$, the minimum degree condition implies that $|W(u)|\ge \frac1c\binom{n-1}{k-1}$ (recall $c:=m^{k-1}$).
So any set of $c+1$ vertices in $V(G)$ contains two vertices that are $(F,\a, 1)$-reachable (here we use that $(c+1)/c-1\ge \binom{c+1}2\r^2$). 
Thus, we can apply Lemma \ref{lem:P} to $G$ to obtain a partition $\cP=\{V_1, \dots, V_r\}$ of $V(G)$ in time $O(n^{2^{c-1}m+1})$. 
Note that $|V_i|\ge (\delta'-\a)n \ge n/m$ for all $i\in [r]$. Also $r\le  1/\delta'\le m$ and each $V_i$ is $(F,\beta, 2^{c-1})$-closed in $H$.
Thus, $\cP$ is an $(F,\beta, 2^{c-1}, 1/m)$-good partition of $V(G)$.

Note that Theorem~\ref{thm:ShZh} shows that $\delta(F, 1, 5m^2)\le 1-1/\chi_{cr}(F)$ and thus $\delta(G)\ge (1-1/\chi_{cr}(F)+\r) n \ge (\delta(F, 1, 5m^2)+\r)n$.
Moreover, Proposition~\ref{prop:group} shows that $|Q(\cP, L_{\cP,F}^{\mu}(G))|\le (2 m-1)^r$.
So by Theorem~\ref{genthm} with $D:=5m^2$ and $q:=(2m-1)^r$, we conclude that $G$ contains a perfect $F$-packing if and only if $(\mathcal P, L^{\mu}_{\mathcal P,F}(G))$ is $(2m-1)^r$-soluble.

{\bf The algorithm.}
Now we state the algorithm and estimate the running time. 
We run the algorithm with running time $O(n^{2^{m^{k-1}-1}m+1})$ provided by Lemma~\ref{lem:P} and obtain a partition $\cP$ of $V(G)$.
By Theorem~\ref{genthm}, it remains to test if $(\mathcal P, L^{\mu}_{\mathcal P,F}(G))$ is $(2m-1)^r$-soluble.
This can be done by testing whether any $F$-packing $M$ of size at most $(2m-1)^r$ is a $q$-solution of $(\mathcal P, L^{\mu}_{\mathcal P,F}(G))$, in time $O(n^{m(2m-1)^r})=O(n^{m(2m-1)^m})$.
If there is a $q$-solution $M$ for $(\mathcal P, L^{\mu}_{\mathcal P,F}(G))$, output YES; otherwise output NO.
The overall running time is $O(n^{\max\{2^{m^{k-1}-1}m+1, \,m(2m-1)^m\}})$.
\end{proof}

\section{Packing $k$-partite $k$-uniform hypergraphs}\label{sec10}

In this section we prove Theorem~\ref{thm:Ktil}. For this we will first collect together a few useful results.
Throughout this section we consider a (not necessarily complete) $k$-partite $k$-graph $F$ on $m$ vertices, and let $a$ be the minimum of the size of the smallest vertex class over all $k$-partite realisations of $V(F)$.
Let $K(F)\supseteq F$ be a complete $k$-partite $k$-graph on $m$ vertices such that the smallest vertex class has $a$ vertices.
We will also write $\sigma (F):=a/m$.

The next proposition is a supersaturation result of Erd\H{o}s \cite{erdos}.

\begin{proposition}\label{erdos}
Let $\eta> 0$, $k, r \in \mathbb{N}$ and let
  $K:=K^{(k)}(a_1, \dots, a_k)$ be the complete $k$-partite $k$-graph  with $a_1\le \cdots\le a_k$ vertices in each class.  there exists $0<\mu \ll \eta$ such that the following holds for sufficiently large $n$. Let $H$ be an $k$-graph on $n$ vertices with a vertex partition $V_1 \cup \dots \cup V_r$. Consider not necessarily distinct $i_1, \dots, i_k\in [r]$. Suppose $H$ contains at least $\eta {n}^{k}$ edges $e=\{ v_1, \dots, v_k \}$ such that $v_1\in V_{i_1}$, $\dots, v_k\in V_{i_k}$. Then $H$ contains at least $\mu {n}^{a_1+\cdots + a_k}$ copies of $K$ whose $j$th part is contained in $V_{i_j}$ for all $j\in [k]$.
\end{proposition}

We also use the following result of Mycroft~\cite[Theorem 1.5]{Mycroft} which forces an almost perfect $F$-packing.

\begin{theorem}\cite{Mycroft}\label{thm:My}
Let $F$ be a $k$-partite $k$-graph.
There exists a constant $D=D(F)$ such that for any $\a>0$ there exists an $n_0=n_0(F,\a)$ such that any $k$-graph $H$ on $n\ge n_0$ vertices with $\delta_{k-1}(H)\ge \sigma(F) n + \a n$ admits an $F$-packing covering all but at most $D$ vertices of $H$.
\end{theorem}

The following proposition shows that $Q(\cP, L_{\cP,F}^{\mu}(H))$ has bounded size.

\begin{proposition}\label{prop:group1}
Let $t,r,k,n_0 \in \mathbb N$ so that $k \geq 3$ and $\beta, \mu, \gamma>0$ so that
\[
1/n_0 \ll \beta, \mu \ll \r, 1/m, 1/t,1/r.
\]
Let $F$ be a $k$-partite $k$-graph on $m$ vertices.
Suppose $H$ is a $k$-graph on $n\ge n_0$ vertices such that $\delta_{k-1}(H)\ge (\sigma(F)+\r) n$ with an $(F,\beta, t, 1/m)$-good partition $\cP$ where $|\cP|=r$. 
Then $|Q(\cP, L_{\cP,F}^{\mu}(H))|\le (2 m-1)^r$.
\end{proposition}

\begin{proof}
Write $L:=L_{\cP,F}^{\mu}(H)$.
It suffices to show that for any element $\bfv\in L_{\max}^{r}$, there exists $\bfv'=(v_1',\dots, v_r') \in L_{\max}^{r}$ such that $-(m-1)\le v_i'\le m-1$ for all $i\in [r]$ and $\bfv - \bfv'\in L$.
In particular, the number of such $\bfv'$ is at most $(2m-1)^r$.

Let $\cP = \{V_1, \dots, V_r\}$ be the partition of $V(H)$ given in the statement of the proposition.
Fix any $i\in [r]$ and consider all edges that contain at least $k-1$ vertices from $V_i$.
Since $\delta_{k-1}(H)\ge (a/m+\r) n$, there are at least $\frac1k\binom{|V_i|}{k-1} (a/m+\r) n$ such edges.
By averaging, there exists $j_i\in [r]$ (it may be that $j_i=i$) such that $H$ contains at least
\[
\frac1r \cdot \frac1k\binom{|V_i|}{k-1} (a/m+\r) n \ge \frac{1}{m^{k} k!r}n^{k}
\]
edges with vertex set $\{ v_1, \dots, v_k \}$ such that $v_1\in V_{j_i}$ and $\{v_2,\dots, v_{k}\}\subseteq V_{i}$. (Here we used $|V_i|\ge n/m$ and $1/n\ll \r$.)
By applying Proposition~\ref{erdos}, since $\mu \ll 1/(m^k k!r)$, we get that there are at least $\mu n^{m}$ copies of $K(F)$ in $H$ whose vertex class of size $a$ is contained in $V_{j_i}$ and other vertex classes are contained in $V_{i}$.
This means that $(m-a)\bfu_i+a\bfu_{j_i}\in L$ for each $i \in [r]$.

Now let $\bfv=(v_1,\dots, v_r)\in L_{\max}^{r}$ and let $l(\bfv):=\sum_{i\in [r]} |v_i|$.
We do the following process iteratively.
For an intermediate step, let $\bfv^*=(v_1^*,\dots, v_r^*)$ be the current vector and take $i\in [r]$ such that $|v_i^*|$ is the maximised over all $i\in [r]$.
We thus subtract $(m-a)\bfu_i+a\bfu_{j_i}$ from $\bfv^*$ if $v_i^*\ge m-a$ or add $(m-a)\bfu_i+a\bfu_{j_i}$ to $\bfv^*$ if $v_i^*\le a-m$.
Note that this process will end because after each step $l(\bfv^*)=\sum_{i\in [r]} |v_i^*|$ decreases by at least $m-2a>0$.
This means that we will reach a vector $\bfv'=(v_1',\dots, v_r')\in L_{\max}^{r}$ such that $-(m-1)\le v_i'\le m-1$ for all $i\in [r]$ and $\bfv - \bfv'\in L$. So we are done.
\end{proof}

\begin{proof}[Proof of Theorem~\ref{thm:Ktil}]
Let $D:=D(F)$ be given by Theorem~\ref{thm:My}.
Given any $\delta \in (\sigma (F),1]$ let $\mu, \alpha, \gamma >0$ so that $0<\mu\ll \alpha\ll \gamma \ll (\delta -\sigma(F)), 1/D, 1/m$.
Apply Lemma \ref{lem:P} with $c:=m$, $\delta':=1/m+\r /2$ to obtain some $\beta>0$. We may assume $\beta \ll \alpha$. Finally choose $n_0 \in \mathbb N$ such that $1/n_0 \ll \beta, \mu$.
Altogether we have
\[
1/n_0\ll \beta , \mu \ll \a \ll \r \ll (\delta-\sigma (F)), 1/D, 1/m.
\]
Let $H$ be an $n$-vertex $k$-graph as in the statement of Theorem~\ref{thm:Ktil}. Note that we may assume that  $n\geq n_0$ and $m$ divides $n$ since else the result is trivial.
We have that $\delta _{k-1} (H) \geq \delta n \geq (\sigma (F)+\gamma)n$.
By Proposition~\ref{prop:deg}, we have $\delta_1(H)\ge \delta {n-1 \choose k-1} \ge (\sigma (F)+\r ) {n-1 \choose k-1}$.


First, for every $v\in V(H)$, we give a lower bound on $|\tilde{N}_{F, \a, 1}(v, H)|$.
Note that for any $(k-1)$-set $S\subseteq V(H)$, we have $|N(S)| \ge (\sigma (F)+\r) n$. Then by Lemma \ref{LM4.2}, for any distinct $u, v \in V(H)$, $u \in \tilde{N}_{F, \a, 1}(v, H)$ if $|N(u) \cap N(v)| \ge \r^2  {n \choose k-1}$. By double counting, we have
\[
{\sum_{S\in N(v)}(|N(S)|-1)} < |\tilde{N}_{F, \a, 1}(v, H)| \cdot |N(v)| + n \cdot \r^2  {n \choose k-1}.
\] 
Note that $|N(v)| \ge \delta_1(H)\ge \delta {n-1 \choose k-1}$. Since $\r \ll \delta,1/k$, we have that
\begin{equation}\label{eq:NK}
|\tilde{N}_{F, \a, 1}(v, H)| > (\sigma (F)+\r) n -1- \frac{\r^2n^k}{|N(v)|} \ge (\sigma (F) +\r/2)n\geq \left (\frac{1}{m}+\frac{\gamma}{2}\right )n.
\end{equation}

Next we claim that every set $A$ of $m+1$ vertices  in $V(H)$ contains two vertices that are $(F, \a, 1)$-reachable in $H$.
Indeed, since $\delta_{1}(H)\ge \delta {n-1 \choose k-1}$, the degree sum of any $m+1$ vertices is at least $(m+1)\delta{n-1 \choose k-1}$. Since $\r \ll 1/m$, we have 
\[
(m+1)\delta {n-1 \choose k-1} > \left(1+{m+1 \choose 2}\r\right){n \choose k-1}.
\] 
Thus, there exist distinct $u, v \in A$ such that $|N(u) \cap N(v)| \ge \r  {n \choose k-1}$, and so they are $(F,\a,1)$-reachable by Lemma~\ref{LM4.2}.

By \eqref{eq:NK} and the above claim, we can apply Lemma \ref{lem:P} to $H$ with the constants chosen at the beginning of the proof.
We get a partition $\cP = \{V_1, \dots, V_r\}$ of $V(H)$ such that $r\le m$ and for any $i\in [r]$, $|V_i|\ge (\sigma (F)+\r/2-\a)n \ge n/m$ and $V_i$ is $(F, \beta, 2^{m-1})$-closed in $H$. 
Thus, $\cP$ is a $(F, \beta, 2^{m-1}, 1/m)$-good partition of $V(H)$.

Note that Theorem~\ref{thm:My} shows that $\delta(F, k-1, D)\le \sigma(F)$ and thus $\delta_{k-1}(H)\ge (\sigma(F)+\r) n\ge (\delta(F, k-1, D)+\r)n$.
Moreover, Proposition~\ref{prop:group1} shows that $|Q(\cP, L_{\cP,F}^{\mu}(H))|\le (2 m-1)^r$.
So by Theorem~\ref{genthm}, with $q:=(2m-1)^r$, we conclude that $H$ contains a perfect $F$-packing if and only if $(\mathcal P, L^{\mu}_{\mathcal P,F}(H))$ is $(2m-1)^r$-soluble.

{\bf The algorithm.}
Now we state the algorithm and estimate the running time. 
We run the algorithm with running time $O(n^{2^{m-1}m+1})$ provided by Lemma~\ref{lem:P} and obtain a partition $\cP$ of $V(H)$.
By Theorem~\ref{genthm}, it remains to test if $(\mathcal P, L^{\mu}_{\mathcal P,F}(H))$ is $(2m-1)^r$-soluble.
This can be done by testing whether any $F$-packing $M$ in $H$ of size at most $(2m-1)^r$ is a $q$-solution of $(\mathcal P, L^{\mu}_{\mathcal P,F}(H))$, in time $O(n^{m(2m-1)^r})=O(n^{m(2m-1)^m})$.
If there is a $q$-solution $M$ for $(\mathcal P, L^{\mu}_{\mathcal P,F}(H))$, output YES; otherwise output NO.
Since $m\ge 3$ and thus $2^{m-1}m+1<m(2m-1)^m$,
the overall running time is $O(n^{m(2m-1)^m})$.
\end{proof}

\section{Concluding remarks}
In this paper we introduced a general structural theorem (Theorem~\ref{genthm}) 
which can be used to 
determine classes 
of (hyper)graphs for which the decision problem for perfect $F$-packings is polynomial time solvable. We then gave three applications of this result.
It would be interesting to find other applications of Theorem~\ref{genthm}.

In light of Conjecture~\ref{conj1} it is likely that one can replace the condition that $\delta ^* =\max \{1/3, c^*_{k,\ell} \}$ in Theorem~\ref{mainthm} with $\delta ^* = c^*_{k,\ell}$.
Theorem~\ref{genthm} is likely to be useful for this. However, note that
in the proof of Theorem~\ref{mainthm}, the condition $\delta ^*\geq 1/3$ ensured that the partition $\mathcal P$ of $V(H)$ consisted of at most $2$ vertex classes.
We then showed that our hypergraph $H$ contained a perfect matching or 
that the coset group $Q$ had bounded size. In particular, since $|\mathcal P|\leq 2$ it was relatively straightforward to show that $|Q|$ was bounded. However, if we no longer have that $\delta ^*\geq 1/3$ we may have that 
$\mathcal P$ consists of many classes. Thus, determining that $Q$ has bounded size is likely to be substantially harder in this case.
 
In Theorems~\ref{mainthm},~\ref{thm:Ftil} and~\ref{thm:Ktil}
we provided algorithms for \emph{determining} whether a hypergraph contains a perfect matching or packing. It would be  interesting to  obtain analogous results which \emph{produce} a perfect matching or packing if such a structure exists.

\section*{Acknowledgment}
The authors are grateful to Richard Mycroft for helpful discussions, particularly concerning~\cite{KKM13}.

\bibliographystyle{plain}

\begin{thebibliography}{10}

\bibitem{AFHRRS}
N.~Alon, P.~Frankl, H.~Huang, V.~R{\"o}dl, A.~Ruci{\'n}ski and B.~Sudakov.
\newblock Large matchings in uniform hypergraphs and the conjecture of {E}rd{\H
  o}s and {S}amuels.
\newblock {\em J. Combin. Theory Ser. A}, 119(6):1200--1215, 2012.

\bibitem{AY96}
N.~Alon and R.~Yuster.
\newblock {$H$}-factors in dense graphs.
\newblock {\em J. Combin. Theory Ser. B}, 66(2):269--282, 1996.

\bibitem{CKO}
O.~Cooley, D.~K{\"u}hn and D.~Osthus.
\newblock Perfect packings with complete graphs minus an edge.
\newblock {\em European J. Combin.}, 28(8):2143--2155, 2007.

\bibitem{CzKa}
A.~Czygrinow and V.~Kamat.
\newblock Tight co-degree condition for perfect matchings in 4-graphs.
\newblock {\em Electron. J. Combin.}, 19(2):Paper 20, 16, 2012.

\bibitem{Edmonds}
J.~Edmonds.
\newblock Paths, trees, and flowers.
\newblock {\em Canad. J. Math.}, 17:449--467, 1965.

\bibitem{erdos}
P.~Erd\H{o}s.
\newblock On extremal problems of graphs and generalized graphs.
\newblock {\em Israel J. Math.}, 2(3):183--190, 1964.

\bibitem{garey}
M.R. Garey and D.S. Johnson.
\newblock {\em Computers and intractability, Freeman}, 1979.

\bibitem{hs}
A. Hajnal and E.~Szemer\'{e}di, Proof of a conjecture of Erd\H{o}s,
\emph{Combinatorial Theory and its Applications vol. II}~{\bf 4}, 601--623, 1970.

\bibitem{HPS}
H.~H\`an, Y.~Person and M.~Schacht.
\newblock On perfect matchings in uniform hypergraphs with large minimum vertex
  degree.
\newblock {\em SIAM J. Discrete Math}, 23:732--748, 2009.



\bibitem{Han14_poly}
J.~Han.
\newblock Decision problem for perfect matchings in dense uniform hypergraphs.
\newblock {\em Trans. Amer. Math. Soc., accepted}.



\bibitem{JieNote}
J.~Han.
\newblock Perfect matchings in hypergraphs and the Erd\H{o}s matching
  conjecture.
\newblock {\em SIAM J. Discrete Math.}, 30:1351--1357, 2016.

\bibitem{Han15_mat}
J.~Han.
\newblock {Near Perfect Matchings in $k$-uniform Hypergraphs II}.
\newblock {\em SIAM J. Discrete Math.}, 30:1453--1469, 2016. 

\bibitem{jieconf}
J.~Han.
\newblock The complexity of perfect packings in dense graphs,
\newblock submitted.


\bibitem{HeKi}
P.~Hell and D.~G. Kirkpatrick.
\newblock On the complexity of general graph factor problems.
\newblock {\em SIAM J. Comput.}, 12(3):601--609, 1983.

\bibitem{HuSc}
C.A.J. Hurkens and A.~Schrijver.
\newblock On the size of systems of sets every {$t$} of which have an {SDR},
  with an application to the worst-case ratio of heuristics for packing
  problems.
\newblock {\em SIAM J. Discrete Math.}, 2(1):68--72, 1989.

\bibitem{Kann}
V.~Kann.
\newblock Maximum bounded {$H$}-matching is {MAX} {SNP}-complete.
\newblock {\em Inform. Process. Lett.}, 49(6):309--318, 1994.

\bibitem{karp}
R.~M. Karp.
\newblock Reducibility among combinatorial problems.
\newblock In {\em Complexity of computer computations ({P}roc. {S}ympos., {IBM}
  {T}homas {J}. {W}atson {R}es. {C}enter, {Y}orktown {H}eights, {N}.{Y}.,
  1972)}, pages 85--103. Plenum, New York, 1972.

\bibitem{KRS10}
M.~Karpi{\'n}ski, A.~Ruci{\'n}ski and E.~Szyma{\'n}ska.
\newblock Computational complexity of the perfect matching problem in
  hypergraphs with subcritical density.
\newblock {\em Internat. J. Found. Comput. Sci.}, 21(6):905--924, 2010.

\bibitem{Kawa}
K.~Kawarabayashi.
\newblock {$K^-_4$}-factor in a graph.
\newblock {\em J. Graph Theory}, 39(2):111--128, 2002.

\bibitem{KKM13}
P.~Keevash, F.~Knox and R.~Mycroft.
\newblock Polynomial-time perfect matchings in dense hypergraphs.
\newblock {\em Adv. in Math.}, 269:265--334, 2015.

\bibitem{Khan1}
I.~Khan.
\newblock Perfect matchings in 3-uniform hypergraphs with large vertex degree.
\newblock {\em SIAM J. Discrete Math.}, 27(2):1021--1039, 2013.

\bibitem{Khan2}
I.~Khan.
\newblock Perfect matchings in 4-uniform hypergraphs.
\newblock {\em J. Combin. Theory Ser. B}, 116:333--366, 2016.


\bibitem{KKMS} 
H.A. Kierstead, A.V. Kostochla, M. Mydlarz and E. Szemer\'edi,
\newblock A fast algorithm for equitable coloring.
\newblock {\em Combinatorica}, 30:217--224, 2010.








\bibitem{Komlos}
J.~Koml{\'o}s.
\newblock Tiling {T}ur\'an theorems.
\newblock {\em Combinatorica}, 20(2):203--218, 2000.

\bibitem{KSS-AY}
J.~Koml{\'o}s, G.~S{\'a}rk{\"o}zy, and E.~Szemer{\'e}di.
\newblock Proof of the {A}lon-{Y}uster conjecture.
\newblock {\em Discrete Math.}, 235(1-3):255--269, 2001.
\newblock Combinatorics (Prague, 1998).

\bibitem{KuOs06soda}
D.~K{\"u}hn and D.~Osthus.
\newblock Critical chromatic number and the complexity of perfect packings in
  graphs.
\newblock In {\em Proceedings of the {S}eventeenth {A}nnual {ACM}-{SIAM}
  {S}ymposium on {D}iscrete {A}lgorithms}, pages 851--859. ACM, New York, 2006.

\bibitem{KO06mat}
D.~K{\"u}hn and D.~Osthus.
\newblock Matchings in hypergraphs of large minimum degree.
\newblock {\em J. Graph Theory}, 51(4):269--280, 2006.

\bibitem{KuOs-survey}
D.~K{\"u}hn and D.~Osthus.
\newblock Embedding large subgraphs into dense graphs.
\newblock In {\em Surveys in combinatorics 2009}, volume 365 of {\em London
  Math. Soc. Lecture Note Ser.}, pages 137--167. Cambridge Univ. Press,
  Cambridge, 2009.

\bibitem{KuOs09}
D.~K{\"u}hn and D.~Osthus.
\newblock The minimum degree threshold for perfect graph packings.
\newblock {\em Combinatorica}, 29(1):65--107, 2009.

\bibitem{KOTo}
D.~K\"uhn, D.~Osthus and T.~Townsend.
\newblock Fractional and integer matchings in uniform hypergraphs.
\newblock {\em European J. Combin.}, 38:83--96, 2014.

\bibitem{KOT}
D.~K{\"u}hn, D.~Osthus and A.~Treglown.
\newblock Matchings in 3-uniform hypergraphs.
\newblock {\em J. Combin. Theory Ser. B}, 103(2):291--305, 2013.

\bibitem{LeGall14}
F.~Le~Gall.
\newblock Powers of tensors and fast matrix multiplication.
\newblock In {\em Proceedings of the 39th International Symposium on Symbolic
  and Algebraic Computation}, ISSAC '14, pages 296--303, New York, NY, USA,
  2014. ACM.

\bibitem{LM1}
A.~Lo and K.~Markstr\"om.
\newblock F-factors in hypergraphs via absorption.
\newblock {\em Graphs  Combin.}, 31(3):679--712, 2015.

\bibitem{MaRu}
K.~Markstr\"{o}m and A.~Ruci\'{n}ski.
\newblock Perfect {M}atchings (and {H}amilton {C}ycles) in {H}ypergraphs with
  {L}arge {D}egrees.
\newblock {\em European J. Comb.}, 32(5):677--687, July 2011.


\bibitem{Mycroft} R. Mycroft,  Packing $k$-partite $k$-uniform hypergraphs, \emph{J. Combin. Theory Ser. A}, 138:60--132, 2016. 


\bibitem{Pik}
O.~Pikhurko.
\newblock Perfect matchings and {$K^3_4$}-tilings in hypergraphs of large
  codegree.
\newblock {\em Graphs Combin.}, 24(4):391--404, 2008.

\bibitem{RR}
V.~R\"odl and A.~Ruci\'nski.
\newblock Dirac-type questions for hypergraphs — a survey (or more problems
  for endre to solve).
\newblock {\em An Irregular Mind}, Bolyai Soc. Math. Studies 21:561--590, 2010.

\bibitem{RRS06}
V.~R\"odl, A.~Ruci\'nski and E.~Szemer\'edi.
\newblock A {D}irac-type theorem for 3-uniform hypergraphs.
\newblock {\em Combinatorics, Probability and Computing}, 15(1-2):229--251,
  2006.

\bibitem{RRS06mat}
V.~R{\"o}dl, A.~Ruci{\'n}ski and E.~Szemer{\'e}di.
\newblock Perfect matchings in uniform hypergraphs with large minimum degree.
\newblock {\em European J. Combin.}, 27(8):1333--1349, 2006.

\bibitem{RRS09}
V.~R{\"o}dl, A.~Ruci{\'n}ski and E.~Szemer{\'e}di.
\newblock Perfect matchings in large uniform hypergraphs with large minimum
  collective degree.
\newblock {\em J. Combin. Theory Ser. A}, 116(3):613--636, 2009.

\bibitem{ShZh}
A.~Shokoufandeh and Y.~Zhao.
\newblock Proof of a tiling conjecture of {K}oml\'os.
\newblock {\em Random Structures Algorithms}, 23(2):180--205, 2003.

\bibitem{Szy13}
E.~Szyma{\'n}ska.
\newblock The complexity of almost perfect matchings and other packing problems
  in uniform hypergraphs with high codegree.
\newblock {\em European J. Combin.}, 34(3):632--646, 2013.

\bibitem{TrZh12}
A.~Treglown and Y.~Zhao.
\newblock Exact minimum degree thresholds for perfect matchings in uniform
  hypergraphs.
\newblock {\em J. Combin. Theory Ser. A}, 119(7):1500--1522, 2012.

\bibitem{TrZh13}
A.~Treglown and Y.~Zhao.
\newblock Exact minimum degree thresholds for perfect matchings in uniform
  hypergraphs {II}.
\newblock {\em J. Combin. Theory Ser. A}, 120(7):1463--1482, 2013.

\bibitem{TrZh15}
A.~Treglown and Y.~Zhao.
\newblock {A note on perfect matchings in uniform hypergraphs}.
\newblock {\em Electron. J. Combin.}, 23:P1.16, 2016.

\bibitem{Tu47}
W.T. Tutte.
\newblock The factorization of linear graphs.
\newblock {\em J. London Math. Soc.}, 22:107--111, 1947.

\bibitem{Yuster07}
R.~Yuster.
\newblock Combinatorial and computational aspects of graph packing and graph
  decomposition.
\newblock {\em Computer Science Review}, 1(1):12 -- 26, 2007.

\bibitem{zsurvey} Y. Zhao. Recent advances on Dirac-type problems for hypergraphs,  \emph{Recent Trends in Combinatorics}, the IMA Volumes in Mathematics and its Applications 159. Springer, New York, 2016. Vii 706.
\end{thebibliography}

\end{document}